\documentclass[11pt]{amsart}

\usepackage{epigamath}

\usepackage[english]{babel}

\usepackage{enumerate,tikz-cd,mathtools,amssymb}

\usepackage[all,cmtip]{xy}

\newtheorem{thm}{Theorem}[section]
\newtheorem{corollary}[thm]{Corollary}
\newtheorem{lemma}[thm]{Lemma}

\newtheorem{proposition}[thm]{Proposition}

\newtheorem*{thm*}{Theorem}
\newtheorem*{corollary*}{Corollary}
\newtheorem*{lemma*}{Lemma}
\newtheorem*{ld*}{Lemma/Definition}
\newtheorem*{proposition*}{Proposition}

\theoremstyle{definition}
\newtheorem{definition}[thm]{Definition}
\newtheorem{remark}[thm]{Remark}
\newtheorem{example}[thm]{Example}

\newtheorem*{definition*}{Definition}
\newtheorem*{example*}{Example}
\newtheorem*{xca*}{Exercise}
\newtheorem*{claim*}{Claim}
\newtheorem*{fact*}{Fact}
\newtheorem*{notation*}{Notation}
\newtheorem*{construction*}{Construction}
\newtheorem*{ack*}{Acknowledgements}
\newtheorem*{question*}{Question}
\newtheorem*{problem*}{Problem}
\newtheorem*{conjecture*}{Conjecture}
\newtheorem*{assumption*}{Assumption}

\newcommand{\C}{\mathbb{C}}
\newcommand{\Z}{\mathbb{Z}}
\newcommand{\Q}{\mathbb{Q}}
\newcommand{\R}{\mathbb{R}}

\newcommand{\pr}[1]{\mathbb P^{#1}}

\newcommand{\HHom}{\mathcal{H}om}

\DeclareMathOperator{\id}{id}
\DeclareMathOperator{\rk}{rk}

\DeclareMathOperator{\SL}{SL}

\DeclareMathOperator{\Hom}{Hom}
\DeclareMathOperator{\Ext}{Ext}

\DeclareMathOperator{\Sym}{Sym}
\DeclareMathOperator{\Ch}{ch}
\DeclareMathOperator{\td}{td}

\DeclareMathOperator{\im}{Im}
\DeclareMathOperator{\Gr}{Gr}

\DeclareMathOperator{\Pic}{Pic\,}
\DeclareMathOperator{\Hilb}{Hilb}

\DeclareMathOperator{\Stab}{Stab}

\DeclareMathOperator{\Coh}{Coh}

\newcommand{\sA}{\mathcal{A}}
\newcommand{\sB}{\mathcal{B}}
\newcommand{\sH}{\mathcal{H}}
\newcommand{\sO}{\mathcal{O}}
\newcommand{\sK}{\mathcal{K}}
\newcommand{\sY}{\mathcal{Y}}

\newcommand{\sL}{\mathcal{L}}

\newcommand{\sS}{\mathcal{S}}
\newcommand{\sT}{\mathcal{T}}

\newcommand{\sM}{\mathcal{M}}
\newcommand{\sC}{\mathcal{C}}
\newcommand{\sD}{\mathcal{D}}

\newcommand{\sE}{\mathcal{E}}
\newcommand{\sF}{\mathcal{F}}
\newcommand{\sG}{\mathcal{G}}

\newcommand{\st}{\,|\,}					

\usepackage{leftidx}						

\usepackage{xcolor}						

\DeclarePairedDelimiter{\set}{\lbrace}{\rbrace}

\DeclarePairedDelimiter{\pair}{\langle}{\rangle}

\DeclarePairedDelimiter{\abs}{\lvert}{\rvert}

\newcommand{\Ku}{\mathsf{Ku}}
\renewcommand{\P}{\mathbb{P}}
\newcommand{\sI}{\mathcal{I}}
\newcommand{\bigslant}[2]{\left.\raisebox{.2em}{$#1$}\middle/\raisebox{-.2em}{$#2$}\right.}

\EpigaVolumeYear{6}{2022} \EpigaArticleNr{13} 
\ReceivedOn{January 5, 2021}
\InFinalFormOn{March 2, 2022}
\AcceptedOn{April 1, 2022}

\title{Moduli spaces on the Kuznetsov component of Fano threefolds of index 2}
\titlemark{Moduli spaces on the Kuznetsov component of Fano threefolds of index 2}

\author{Matteo Altavilla}
\address{Department of Mathematics, University of Utah, Salt Lake City,
  UT 84102, USA} 
\email{altavilla@math.utah.edu}

\author{Marin Petkovi\'{c}}
\address{Department of Mathematics, University of Utah, Salt Lake City,
UT 84102, USA} 
\email{petkovic@math.utah.edu}

\author{Franco Rota}
\address{School of Mathematics and Statistics, University of Glasgow, University Place, Glasgow,
G12 8QQ, UK} 
\email{franco.rota@glasgow.ac.uk}

\authormark{M. Altavilla, M. Petkovi\'{c}, and F. Rota}

\AbstractInEnglish{General hyperplane sections of a Fano threefold $Y$ of index 2 and Picard rank 1 are del Pezzo surfaces, and their Picard group is related to a root system. To the corresponding roots, we associate objects in the Kuznetsov component of $Y$ and investigate their moduli spaces, using the stability condition constructed by Bayer, Lahoz, Macr\`i, and Stellari, and the Abel--Jacobi map. 
We identify a subvariety of the moduli space isomorphic to $Y$ itself, and as an application we prove a (refined) categorical Torelli theorem for general quartic double solids.}



\acknowledgement{All three authors were partially supported by NSF-FRG grant DMS 1663813, PI: Aaron Bertram.}

\begin{document}



\maketitle

\begin{prelims}

\DisplayAbstractInEnglish

\bigskip

\DisplayKeyWords

\medskip

\DisplayMSCclass







\end{prelims}


\newpage

\setcounter{tocdepth}{1}

\tableofcontents


\section{Introduction}\label{sec_intro}

Stability conditions on triangulated categories were introduced by Bridgeland in \cite{Bri07_triang_cat}. 
A remarkable feature of Bridgeland's construction is that the set $\Stab(\sD)$ of all stability conditions admits a natural topology, and is endowed with the structure of a complex manifold \cite[Theorem~1.2]{Bri07_triang_cat}. Even when $\sD=D^b(\Coh(X))$ is the bounded derived category of a smooth projective variety $X$, it is challenging to study $\Stab(\sD)$, and questions about its non-emptiness, connectedness, or simple-connectedness are difficult problems. 

 
On the other hand, stability conditions allow for the construction of moduli spaces of stable objects of $\sD$. It is interesting to investigate properties of these moduli spaces, like non-emptiness, irreducibility, smoothness, or projectivity, or to relate them to other classical moduli spaces: for example, Bridgeland stability regulates the birational geometry of the Hilbert scheme of points of surfaces (see \cite{ABCH13} for the case of $\pr 2$, and \cite{BC13} for other surfaces), or of moduli spaces of sheaves on a K3 surface \cite{BM14_K3}.

If $\sD$ admits an exceptional collection, Bayer, Lahoz, Macr\`i, and Stellari give a sufficient condition to induce a stability condition on the right orthogonal of the collection, starting from a (weak) stability condition on $\sD$ \cite{BLMS17}. The criterion applies for example to cubic fourfolds and Fano threefolds of Picard rank one \cite{BLMS17} and to Gushel-Mukai varieties \cite{PPZ19}.

\medskip

In this paper, we consider smooth Fano threefolds $Y$ of Picard rank 1 and index 2: they belong to one of 5 families indexed by their degree $d\in\{1,\ldots,5\}$.
The derived category of $Y$ admits a semi-orthogonal decomposition $ D(Y)=\langle \Ku(Y), \sO_Y, \sO_Y(1) \rangle $ whose non-trivial part $\Ku(Y)$ is called the \textit{Kuznetsov component of $Y$}. Its \emph{numerical Grothendieck group} is a rank 2 lattice spanned by two classes $v$ and $w$ (see Section~\ref{sec_Fano_threefolds} for the details). 

The focus of this work is on moduli spaces of complexes of class $w$, with particular attention to the degree 2 and 1 cases.
A more detailed study of moduli of class $w$ for cubic threefolds appears in \cite{Bayer}; moduli of objects of class $v$ are studied in \cite{PY20} (for degrees $d\geq 2$) and in \cite{PR20_veronese} for the case of degree 1.

\subsection*{Summary of the results}

The main result of our work is a description of the moduli space $\sM_\sigma(w)$, which parametrizes complexes in $\Ku(Y)$ of class $w$, semistable with respect to one of the stability conditions $\sigma$ constructed in \cite{BLMS17}. Throughout the paper, we work under some mild generality assumptions for $Y$, to control the singularities that may appear in the hyperplane sections of $Y$ (see Section~\ref{sec_hyperplane sections+generality}) and to rely on previous results of \cite{Wel81,Tih82}.

First, we study the moduli spaces of Gieseker-stable sheaves on $Y$ of class $w$. By a \emph{root} of $Y$ we mean a sheaf of the form  $\iota_*(\sO_S(D))$, with $S\in \abs{\sO_Y(1)}$ and $D\in \Pic(S)$ satisfying $D^2=-2$ and $D\cdot H_{|S}=0$ (see Definition~\ref{def_of_root_of_S}). We have:

\begin{proposition}[\emph{cf.} Proposition~\ref{gieseker}]
Let $Y$ be a general Fano threefold of Picard rank 1 and index 2. The moduli $\sM_H(w)$ of Gieseker-semistable objects of class $w$ on $Y$ has two irreducible components. 

One component, denoted $\mathcal{P}$, has dimension $d+3$ and parametrizes ideal sheaves  $I_{p|S}$ of a point in a hyperplane section $S$ of $Y$. A point $[I_{p|S}]\in \mathcal P$ is smooth in $\sM_H(w)$ iff $p$ is a smooth point of $S$.
The second component has dimension $d+1$ and its smooth points parametrize roots of $Y$.

Both components are smooth outside their intersection, which is the locus parametrizing $I_{p|S}$, with $S$ singular at~$p$.
\end{proposition}

We relate $\sM_H(w)$ and $\sM_\sigma(w)$ via wall-crossing: there is a two-parameter family $\sigma_{\alpha,\beta}$ of weak stability conditions on $D^b(Y)$. We show that the moduli space of $\sigma_{\alpha,\beta}$-stable complexes of class $w$ is isomorphic to $\sM_H(w)$ (resp. to $\sM_\sigma(w)$) for $\beta=-\frac12$ and $\alpha \gg 0$ (resp. $0< \alpha \ll 1 $) in Proposition~\ref{prop_tilt} and Proposition~\ref{kuzistilt}. 
Then, we deform the parameter $\alpha$: there is a unique wall separating $\sM_H(w)$ and $\sM_\sigma(w)$ and the universal family of $\sM_H(w)$ induces a wall-crossing morphism, defined at the level of sets by replacing objects $I_{p|S}$ 
with complexes defined by triangles 
\begin{equation}\label{eq_def_of_Ep}
\sO(-1)[1] \to E_p \to I_p
\end{equation} (Proposition \ref{prop_GiesekerVSTilt}). In this way we obtain a general description for all $d$:

\begin{thm}[\emph{cf.} Theorems \ref{thm_Section2}, \ref{thm_moduliY2}, and \ref{thm_moduliY1}]\label{introduction_1}
Let $Y$ be a general Fano threefold of Picard rank 1 and index 2. Then, the moduli space $\sM_\sigma(w)$ is proper. It contains two smooth subvarieties $\sY$ and $\sC$. 
The locus $\sY$ is isomorphic to $Y$ and it parametrizes objects $E_p$ of the form \eqref{eq_def_of_Ep}, while the locus $\sC$ has dimension $d+1$ and parametrizes roots of $Y$. 

For $d\geq 3$, $\sM_\sigma(w)$ is projective and irreducible and $\sY$ is contained in the closure of $\sC$; $\sY$ and $\sC$ are distinct irreducible components for degree $d\leq 2$. 
\end{thm}

For $d\leq 2$, we consider the intermediate Jacobian $J(Y)$ of $Y$ and study the Abel--Jacobi map
\[\Psi \colon \mathcal{M}_\sigma(w) \to J(Y)\]
sending a complex to its second Chern class (see Section \ref{sec_AJ_map} for the details). 


In degree 1, we show that $\sC\simeq F(Y)$, the Fano surface of lines of $Y$, and the intersection $\sY\cap F(Y)$ in $\sM_\sigma(w)$ is a curve $C$. The map $\Psi$ is an embedding outside $C$, and its image suffices to determine $Y$ \cite{Tih82}.

In degree 2, $Y$ is a double cover of $\pr 3$ ramified over a quartic K3 surface $R$. The intersection $\sY\cap \sC$ is isomorphic to $R$ (see Theorem \ref{thm_moduliY2}), $\Psi$ contracts the component $\sY$ to a point, and it is a generic embedding on the component $\sC$ (Corollary~\ref{cor_C_generically_embedded_by_AJ}).
Still in the case of degree 2, we apply our result to show a (refined) categorical Torelli theorem: 

\begin{thm}[\emph{cf.} Theorem \ref{thm_torelliY2}]\label{introduction_3}
Let $Y$ and $Y'$ be two general quartic double solids. There exists an equivalence $u \colon \Ku(Y') \simeq \Ku(Y)$ if and only if $Y'\simeq Y$.
 \end{thm}

Our proof technique is inspired by that of \cite{BMMS12}; we use the equivalence $u$ to construct an isomorphism between moduli spaces and argue that this is sufficient to conclude.

\subsection*{Related work and further questions}


In the case of degree 5, Theorem \ref{introduction_1} recovers the description of $Y_5$ within the Hilbert scheme of three points on $\pr 2$ \cite{Muk92}. If the degree is 4, $Y$ is the intersection of 2 quadrics: its Kuznetsov component is equivalent to the derived category of a genus 2 curve $C$ \cite{BO95} and Theorem \ref{introduction_1} recovers results of Reid's \cite{Rei72}, who shows that the Fano surface of lines on $Y$, the intermediate Jacobian of $Y$ (introduced in \cite{Clemens1972}), and the Jacobian variety of $C$ are isomorphic.

The main categorical techniques we use are developed in \cite{BLMS17, Kuz11_base_change}.

The general expectation is that the Kuznetsov component contains sufficient information to recover $Y$.
Reconstruction theorems are known to hold in a few cases: for example, $\Ku(X)$ determines $X$ if $X$ is a cubic fourfold (see \cite{BLMS17} for general $X$, and \cite{LPZ18} for all $X$) or if $X$ is an Enriques surface \cite{LNSZ19,LSZ21}.

In the case of Fano threefolds, four of the five families considered satisfy a refined Torelli theorem: $Y_5$ is rigid in moduli, so the statement is vacuously true. Degrees 4 and 3 are showed in \cite{BO95} and \cite{BMMS12} (\cite{PY20} gives an alternative proof of the cubic case). For degree 2, our result (Theorem \ref{introduction_3}) strengthens the statement of Bernardara--Tabuada \cite[Corollary~3.1(iii)]{BT16}, who show that the same holds under the additional assumption that the equivalence is of Fourier--Mukai type. 

\medskip

There are some natural questions which remain unanswered by this work.
\begin{enumerate}[(i)]
    \item (\textit{Torelli theorem for $d=1$}) It is expected that an analogue of Theorem \ref{introduction_3} holds in degree 1 as well. Unfortunately, the heart of the stability condition $\sigma$ has homological dimension 3 (instead of 2) in degree 1 \cite[Remark 2.6]{PR20_veronese}, so a different technique will be needed.
    \item (\textit{Equivalences of Fourier--Mukai type}) A conjecture of Kuznetsov \cite[Conjecture~3.7]{kuz07HPD} implies, in the case of quartic double solids, that every equivalence $u\colon \Ku(Y')\xrightarrow{\sim}\Ku(Y)$ is of Fourier--Mukai type (\textit{i.e.} that the composition $D(Y')\to \Ku(Y')\xrightarrow{u}\Ku(Y)\to D(Y) $ is a Fourier--Mukai functor). Combining Theorem \ref{introduction_3} with \cite[Proposition~3.5]{BT16} shows that $\Ku(Y')$ and $\Ku(Y)$ are equivalent if and only if they are Fourier--Mukai equivalent. However, this is not sufficient to prove the statement of  \cite[Conjecture~3.7]{kuz07HPD} for degree 2. 
    Interestingly, this conjecture is incompatible with Kuznetsov's conjecture on Fano threefolds (see \cite[Conjecture~3.7]{Kuz09_threefolds} for the conjecture, and \cite[Theorem 4.2]{BT16} for a proof of incompatibility). The conjecture on Fano threefolds has been recently disproven in \cite{Zha20}.

\end{enumerate}

\subsection*{Structure of the paper}
Section \ref{sec_preliminaries} recollects preliminary categorical notions, such as (weak) stability conditions (\ref{appendixstability}, \ref{sec_weak_stab_D(X)}), and base change of semi-orthogonal decompositions (\ref{sec_basechange_mutations}). 
Section \ref{sec_Fano_threefolds} is a brief description of smooth Fano threefolds of Picard rank 1 and index 2 and of the geometry of their hyperplane sections (\ref{sec_hyperplane sections+generality}). 

In Section \ref{sec_construction_of_moduli} we describe the moduli spaces of Gieseker-stable sheaves of class $w$ (Proposition~\ref{gieseker}), of $\sigma_{\alpha,\beta}$-stable complexes in $D^b(Y)$ (Proposition~\ref{prop_GiesekerVSTilt}) and of $\sigma$-stable complexes in $\Ku(Y)$ (Theorem \ref{thm_Section2}). We conclude with the definition, and some properties, of the Abel--Jacobi map $\Psi$ (Section \ref{sec_AJ_map}).

Section \ref{sec_higher_degrees} illustrates Theorem \ref{thm_Section2} on a case by case basis, comparing it with results known in the literature, and providing more details to the description of $\sM_\sigma(w)$ in the remaining cases of degree $d=2,1$.

Finally, Section \ref{sec_Torelli} contains the proof of the Torelli theorem for quartic double solids (Theorem \ref{thm_torelliY2}).

%

\subsection*{Acknowledgements} We are grateful to our doctoral advisor, Aaron Bertram, for his guidance and his enthusiasm. We wish to thank Arend Bayer for encouraging us in pursuing this problem. We also thank Paolo Stellari and Huachen Chen, for discussing this problem with us on various occasions. This project benefited from the participation of the second and third author to the workshop \emph{``Semi-orthogonal decompositions, stability conditions and sheaves of categories''} held in Toulouse in 2018 and of the first and third author to the workshop on \emph{``Derived Categories, Moduli Spaces and Deformation Theory''} held in Cetraro in 2019. We thank the organizers and the participants of these events. We are grateful to an anonymous referee for the invaluable feedback on a first version of this work. We thank Laura Pertusi, Song Yang, and Shizhuo Zhang for our conversations on these topics.

\section{Preliminaries}\label{sec_preliminaries}

\subsection{Weak stability conditions} \label{appendixstability}
In this section we refer to \cite{BLMS17} and briefly summarize definitions and results on (weak) stability conditions.

Let $\mathcal{D}$ be a triangulated category and $K(\sD)$ its Grothendieck group; fix a finite rank lattice $\Lambda$ and a surjective homomorphism $v \colon K(\sD) \twoheadrightarrow \Lambda$.
\begin{definition}
A full abelian subcategory $\mathcal{A}\subset \sD$ is called a \emph{heart of a bounded t-structure} if the following holds:
\begin{enumerate}[(a)]
    \item For all $E,\, F\in \mathcal{A}$ and $n<0$ we have $\Hom(E,F[n])=0$;
    \item For every $E\in \sD$ there exist a filtration, \textit{i.e.} objects $E_i\in \sD$, integers $k_1 > \dots > k_m$ and triangles
    \[
\xymatrix{0=E_0 \ar[r] & E_1\ar[d] \ar[r] &E_2 \ar[d]\ar[r] &\dots \ar[r] &E_{m-1} \ar[d]\ar[r]&E_m=E \ar[d]
\\ & A_1[k_1] \ar@{-->}[ul]& A_2[k_2] \ar@{-->}[ul]& & A_{m-1}[k_{m-1}] \ar@{-->}[ul]& A_m[k_m] \ar@{-->}[ul]}
\]
such that $A_i\in \mathcal{A}$; the $A_i$ objects are called the \emph{cohomologies} of $E$ and are denoted by $\sH^{-k_i}_{\sA}(E)$.
\end{enumerate}
\end{definition}

A \emph{weak stability function} on $\sA$ is a group homomorphism $Z\colon K(\mathcal{A})\to \C$ such that for all $E\in \sA$ we have $\Im Z(E) \geq 0$ and $\Im Z(E)=0$ implies $\Re Z(E) \leq 0$. A weak stability function $Z$ is a \emph{stability function} if moreover $\Im Z(E)=0$ implies $\Re Z(E) < 0$ for all $E\neq 0$.

A weak stability function has an associated slope:
\[
\mu_Z(E) = \begin{cases} -\dfrac{\Re Z(E)}{\Im Z(E)} \quad \text{if } \ \Im Z(E) > 0 \\ +\infty \quad \text{otherwise}\end{cases},
\]

\begin{definition} Let $\sD$ be a triangulated category, and let $\Lambda$ and $v$ be as above; a \emph{weak stability condition} on $\sD$ is a pair $\sigma=(\sA, Z)$ where $\sA$ is the heart of a bounded t-structure and $Z$ is a group homomorphism $Z \colon \Lambda \to \C$ , satisfying the following properties:
\begin{enumerate}[(a)]
    \item The composition $Z \circ v \colon K(\sA)=K(\sD) \to \Lambda \to \C$ is a weak stability function on $\sA$;
    \item All $E\in \sA$ have a Harder--Narasimhan filtration with factors $F_i\in \sA$ semistable with respect to $Z$, with strictly decreasing slopes; 
    \item There exists a quadratic form $Q$ on $\Lambda\otimes \R$ that is negative definite on $\ker Z$, such that $Q(v(E))\geq 0$ for $E$ semistable.
\end{enumerate}
If $Z$ is a stability function, then the pair $(\sA,Z)$ is a \emph{Bridgeland stability condition}.
\end{definition}
\begin{definition}

We say that an object $E\in\sA$ is \emph{(semi)stable} with respect to $\sigma$ if for every quotient $E\twoheadrightarrow Q$ in $\sA$ with $0\neq Q\neq E$ we have $\mu_Z(E)\leq\mu_Z(Q)$ (resp. $\mu_Z(E)<\mu_Z(Q)$).

An object $E\in\sD$ is called (semi)stable if there exists $m\in \Z$ such that $E[m]\in\mathcal{A}$ and $E[m]$ is (semi)stable, in which case $-m+\frac{1}{\pi}\arg Z(E)$ is the \emph{phase} of $E$. 
\end{definition}


\begin{definition}\label{tiltofheart}
Given a weak stability condition $\sigma$ and a real number $\mu$, the extension closure defined by $\sA^{\mu, \sigma} \coloneqq \langle \mathcal{T}^{\mu,\sigma}, \mathcal{F}^{\mu,\sigma}[1] \rangle$ where 
\begin{align*}
&\mathcal{T}^{\mu,\sigma}=\{E \in \sA \ | \text{ All HN factors $F$ of $E$ have slope $\mu_Z(F)>\mu$} \} \\
&\mathcal{F}^{\mu,\sigma}=\{E \in \sA \ | \text{ All HN factors $F$ of $E$ have slope $\mu_Z(F)\leq\mu$} \}
\end{align*}
is the heart of a bounded t-structure called a \emph{tilt} of $\sA$ \cite{HRS96,BLMS17}.
\end{definition}



\subsection{Weak stability conditions on $D^b(X)$}\label{sec_weak_stab_D(X)}

Let $X$ be a smooth projective variety of dimension $n$, and let $H$ be an ample divisor on $X$. We follow the setup of \cite[Section~2]{BLMS17} and recall the construction of weak stability conditions on $D^b(X)$.

For $j\in\set{0,\ldots,n}$ define lattices $\Lambda_H^j\simeq \Z^{j+1}$ generated by 
\[ \left(H^n \Ch_0, H^{n-1}\Ch_1,\ldots,H^{n-j}\Ch_j \right) \subset \Q^{j+1} \]
with the sujective map $v_H^j\colon K(X) \to \Lambda_H^j$ given by the Chern character. 
Then, Mumford slope gives rise to a weak stability condition $\sigma_H \coloneqq (\Coh(X),Z_\mu)$ where $Z_H\colon \Lambda_H^1\to \C$ is defined by 
\[ Z_H(E)\coloneqq -H^{n-1}\Ch_1(E) + i H^n\Ch_0(E). \]

We can define a new weak stability condition from the one above via tilting: for $\beta\in \R$, define $\Coh^\beta(X)$ to be the heart obtained from $\sigma_H$ via tilting at slope $\mu_H=\beta$. Before introducing the central charge, recall the notation $\Ch^\beta(E) \coloneqq  \Ch(E) \cdot e^{-\beta H}$. For convenience, we make the first three terms explicit:
\begin{equation*}
    \begin{split}
        \Ch_0^\beta = \Ch_0, \qquad 
        \Ch_1^\beta = \Ch_1 - \beta H\Ch_0, \qquad
        \Ch_2^\beta = \Ch_2 - \beta H \Ch_1 + \frac{\beta^2H^2}{2}\Ch_0.
    \end{split}
\end{equation*}

We have the following:

\begin{proposition}[{\cite[Proposition~2.12]{BLMS17}}] \label{def_tilt_stability}
For all $(\alpha,\beta)\in \R_{>0}\times \R$, the pair given by  $\sigma_{\alpha,\beta}=(\Coh^\beta(X), Z_{\alpha,\beta})$ with
\[
Z_{\alpha,\beta}(E) = -H^{n-2}\Ch_2^\beta(E) + \dfrac{\alpha^2}{2}H^n\Ch_0^\beta(E) + i H^{n-1}\Ch_1^\beta(E)
\]
defines a weak stability condition on $D^b(X)$. Moreover, these weak stability conditions vary continuously with $(\alpha,\beta)$.
\end{proposition}

In virtue of the proposition above, we will consider the upper half plane 
$\set{(\alpha,\beta)\in \R^2  \st \alpha >0 }$ as a space parametrizing the weak stability conditions $\sigma_{\alpha,\beta}$. We write $\mu_{\alpha,\beta}\coloneqq \mu_{Z_{\alpha,\beta}}$ for the slope function associated to $\sigma_{\alpha,\beta}$.  We can now define \textit{walls} and \textit{chambers}:

\begin{definition}\label{def_walls}
Fix a vector $v\in \Lambda^2_H$. 

A \textit{numerical wall} for $v$ is the set of pairs $(\alpha,\beta)\in \R_{>0}\times \R$ such that there exists a vector $w\in \Lambda^2_H$ for which $\mu_{\alpha,\beta}(v)=\mu_{\alpha,\beta}(w)$.

A \textit{wall} for $F\in \Coh^\beta(X)$ is a numerical wall for $v=\Ch_{\leq 2}(F)$ such that for every $(\alpha,\beta)$ in the numerical wall there exists a short exact sequence of semistable objects $0\to E\to F\to G\to 0$ satisfying $\mu_{\alpha,\beta}(E)=\mu_{\alpha,\beta}(F)=\mu_{\alpha,\beta}(G)$.

A \textit{chamber} is a connected component of the complement of the union of walls in $\R_{>0}\times\R$.
\end{definition}

The key feature of this wall and chamber decomposition is that walls with respect to a fixed $v\in \Lambda^2_H$ are locally finite in the upper half-plane. In particular, if $v=\Ch_{\leq 2}(E)$ for some $E\in \Coh^\beta(X)$, stability of $E$ is constant as $(\alpha,\beta)$ varies within a chamber (see \cite[Proposition~B.5]{BMS16}). 

We conclude recalling another construction of a weak stability condition. In the notation of Definition \ref{tiltofheart}, fix $\mu \in \R$ and define $\Coh^{\mu}_{\alpha, \beta}(X) := \Coh^{\beta}(X)^{\mu,\sigma_{\alpha,\beta}}$ endowed with a stability function
\[
Z^\mu_{\alpha,\beta} := \dfrac{1}{u} Z_{\alpha,\beta},
\]
where $u\in \C$ is the unit vector in the upper half plane such that $\mu = -\dfrac{\Re u}{\Im u}$.

\begin{proposition}[{\cite[Proposition~2.15]{BLMS17}}]\label{tiltoftilt}
The pair $\sigma^{\mu}_{\alpha,\beta}= (\Coh^{\mu}_{\alpha, \beta}(X), Z^\mu_{\alpha,\beta})$ is a weak stability condition on $D(X)$.
\end{proposition}

Finally, again in \cite{BLMS17}, the authors induce a stability condition on the Kuznetsov component of $Y$ using weak stability conditions:
\begin{proposition}[{\cite[Theorem~1.1]{BLMS17}}]
\label{thm_stab_con_on_Ku}
Let $Y$ be a Fano threefold of index 2 with $\Pic(Y)=\Z$, and let $\sigma^0_{\alpha,-\frac{1}{2}}$ be the weak stability condition constructed in Proposition \ref{tiltoftilt} for $\mu=0$, $\beta=-\dfrac{1}{2}$ and $0<\alpha \ll 1$; let moreover $\sA= \Ku(Y) \cap \Coh^{0}_{\alpha, \beta}(Y)$ and $Z = Z^0_{\alpha,-\frac{1}{2}}$. Then the pair $\sigma=(\sA, Z)$ is a Bridgeland stability condition on $\Ku(Y)$.
\end{proposition}

\subsection{Projection functors and base change}
\label{sec_basechange_mutations}

The following construction of \cite{Kuz11_base_change} is crucial to some parts of this paper. For any quasi-projective variety $T$, we will use $D(T)$ to denote the category of perfect complexes on $T$. Let $Y$ be a smooth projective variety. 

Suppose $E$ is an exceptional object on $Y$. Then, the category \[\pair{E}_T\coloneqq \pair{E\boxtimes G \st G\in D(T)}\subset D(Y\times T)\] gives rise to projection functors $\mathbb R_E^T$ and $\mathbb L_E^T$ to $\!^\perp\pair{E}_T$ and $\pair{E}_T^\perp$ respectively (\cite[Section~2.3]{Kuz11_base_change}). They are defined as the cones of the (co)evaluation morphisms.

Both $\pair{E}_T$ and the projection functors are \emph{$T$-linear} (this means that they are preserved under tensoring with pull-backs from $D(T)$) \cite[\S 2.3, Corollary~5.9]{Kuz11_base_change}. As a consequence, they commute with base change, in the sense of \cite[\S 5]{Kuz11_base_change}. In particular,  at every closed point $t\in T$, the projection functors $\mathbb R_E^t$ and $\mathbb L_E^t$ are the \textit{right} and \textit{left mutation} of objects on $D(Y\times \set{t})\simeq D(Y)$ across the exceptional object $E$:
\[
 E\otimes \Hom^\bullet(E,X) \xrightarrow{\mathrm{ev}} X \to \mathbb L_E^t(X) \quad \text{and}\quad
  \mathbb R_E^t(X) \to X \xrightarrow{\mathrm{coev}} E\otimes \Hom^\bullet(X,E)^\vee. 
\]

\section{Fano threefolds of Picard rank 1 and index 2}
\label{sec_Fano_threefolds}

Let $Y$ be a smooth Fano threefold of Picard rank 1 and index 2  (\textit{i.e.}, $H\coloneqq -K_Y/2$ is an ample generator of $\Pic(Y)$). Then, \cite{Isk77} shows that $Y$ belongs to one of the following 5 families, indexed by their \textit{degree} $d\coloneqq H^3\in\set{1,\ldots,5}$:
\begin{itemize}
\item $Y_5= \Gr(2,5) \cap \pr 6 \subset \pr 9$ is a linear section of codimension 3 of the Grassmannian $\Gr(2,5)$ in the Pl\"ucker embedding;
\item $Y_4 = Q \cap Q' \subset \pr 5$ is the intersection of two quadric hypersurfaces;
\item $Y_3 \subset \pr 4$ is a cubic hypersurface;
\item $Y_2 \xrightarrow{\pi} \pr 3$ is a double cover ramified over a quartic surface, or equivalently a hypersurface of degree 4 in the weighted projective space $\mathbb P(1,1,1,1,2)$;
\item $Y_1$ is a a double cover branched over a cubic of the cone over the Veronese surface in $\pr 5$, or equivalently a hypersurface of degree 6 in the weighted projective space $\mathbb P(1,1,1,2,3)$.
\end{itemize}

Kuznetsov shows in a series of papers that the derived category of $Y$ admits a semi-orthogonal decomposition $D(Y)=\langle \Ku(Y), \sO_Y, \sO_Y(1) \rangle$, whose non-trivial part 
\begin{equation}
    \label{eq_def_of_Ku}
\Ku(Y)\coloneqq \set{E\in D(Y) \st \Hom(\sO_Y,E[i])=\Hom(\sO_Y(1),E[i])=0 \mbox{ for all }i\in \Z}  
\end{equation}
we call the \textit{Kuznetsov component} of $Y$ (\cite{Kuz09_threefolds, Kuz14}). As observed in \cite[Proposition~3.9]{Kuz09_threefolds}, the numerical Grothendieck group $K_{\text{num}}(\Ku(Y))$ is identified with the image of the Chern character map $\Ch: K(\Ku(Y))\to H^\ast(Y,\Q)$, so we may alternatively use numerical classes and Chern characters. The lattice $K_{\text{num}}(\Ku(Y))$ has rank 2 and is spanned by the classes 
\begin{equation}
    \label{eq_def_of_v_w}
v\coloneqq \left[ 1-\dfrac{1}{d}H^2\right] \qquad \mbox{and} \qquad w\coloneqq \left[ H -\dfrac{1}{2}H^2+\left(\dfrac{1}{6}-\dfrac{1}{d}\right)H^3  \right]
\end{equation}
with a bilinear form given by the Euler pairing
\[
\begin{bmatrix}
-1 & -1 \\
1-d & -d
\end{bmatrix}.
\]

Before moving on to the study of $\sM_\sigma(w)$, we fix notation and recall some features of $Y$ and $\Ku(Y)$ on a case-by-case basis. Here, we also make our genericity assumptions explicit: assuming that $Y$ is smooth suffices to prove our results if $d=3,4,5$, but we need (mild) additional assumptions for lower degrees.

\subsubsection*{Degree 5}
The Fano threefold of Picard rank 1 and degree 5 is often denoted $V_5$ and called the \textit{quintic del Pezzo threefold} (it is unique up to isomorphism \cite{Isk77}). It arises as the intersection of a Grassmannian $\Gr(2,5)\subset \pr 9$ with three hyperplanes of $\pr 9$ \cite[Proposition~10]{Muk92}. The universal sub-bundle and quotient on $\Gr(2,5)$ restrict to exceptional vector bundles $S,Q$ on $V_5$, moreover, it is shown in \cite{Or91} that $V_5$ has a full strong exceptional collection
$D^b(V_5)=\pair{S,Q^\vee,\sO,\sO(1)}$,
which in turn gives rise to a semi-orthogonal decomposition 
\[ \Ku(V_5)\simeq \pair{S,Q^\vee}. \]
The space $\Hom(S,Q^\vee)$ is three-dimensional. Therefore, the category $\Ku(V_5)$ is equivalent to the derived category of representations of the Kronecker quiver with three arrows. 

\subsubsection*{Degree 4}
$Y_4$ is the complete intersection of two quadric hypersurfaces in $\pr 5$. The corresponding pencil of quadrics degenerates at 6 points, and thus determines a genus 2 curve $C$. 
As it turns out (see \cite[Theorem~2.9]{BO95}), $C$ is a fine moduli space for spinor bundles on $Y$, 
and the universal family induces a fully faithful Fourier-Mukai functor $\phi\colon D^b(C) \to D^b(Y_4)$,
giving rise to an equivalence $D^b(C)\simeq \Ku(Y_4)$.



\subsubsection*{Degree 3} The case of a cubic threefold $Y_3$ is treated in \cite{Bayer}. We give a brief account of the results in Section \ref{ssec_deg3}.

\subsubsection*{Degree 2}
We assume that $Y_2 \xrightarrow{\pi} \P^3$ is a general double cover ramified over a quartic K3 surface $R$ (which will be often identified with the branching locus). In this case, $\sO_{Y_2}(1)$ is the pullback of $\sO_{\pr 3}(1)$. The generality assumption here is that of \cite{Wel81,TM03}. Precisely, we assume $R$ to be smooth and to not contain lines \cite[\S 1]{Wel81}, consequently, $Y_2$ and its Fano surface of lines are smooth \cite[Remark 2.2.9]{KPS18_hilbert_schemes}. 


Finally, $Y_2$ is equipped with an involution $\tau$ that swaps the two sheets of the cover, acting on objects in the derived category via pull-back.
Observe that $\tau$ preserves $\Ku(Y_@)$, since $\sO_{Y_2}(1)$ is pulled-back from $\pr 3$. Moreover, we can follow \cite{Kuz15_fractional} and compute the Serre functor for $\Ku(Y_2)$ to be $S_{\Ku(Y_2)}(E) = \tau E[2]$.

\subsubsection*{Degree 1} 

Finally, $Y_1$ is a sextic hypersurface in the weighted projective space $\mathbb P\coloneqq \pr{}(1,1,1,2,3)$ and $\Pic(Y_1)$ is generated by $H\coloneqq  \sO_{\mathbb P}(1)_{|Y_1}$. $H$ has three linearly independent sections and a unique base point $y_0$, hence it induces a rational map $\phi_H\colon Y_1 \dashrightarrow \mathbb{P}(H^0(\sO_{Y_1}(1)))\simeq \pr 2$. On the other hand, $2H\sim -K_{Y_1}$ is base point free, and induces a morphism $\phi_{2H}\colon Y_1 \to \mathbb{P}(H^0(\sO_{Y_1}(2)))\simeq\pr 6$, whose image $K\simeq \mathbb P(1,1,1,2)$ is the cone over a Veronese surface with vertex $k\coloneqq \phi(y_0)$. The morphism $\phi_{2H}$ is smooth of degree 2 outside $k$ and a divisor $D\in \abs{\sO_K(3)}$. For this reason, $Y_1$ is often referred to as to a \textit{Veronese double cone}. We will denote by $\iota$ the involution on $Y_1$ corresponding to the double cover $\phi_{2H}$.

There is a commutative diagram
\begin{equation}
\label{eq_diagram_rational_maps}
\begin{tikzcd}
	{Y_1} & {K} \\
	{\pr 2}
	\arrow["{\phi_H}"', from=1-1, to=2-1, dashed]
	\arrow["{\phi_{2H}}", from=1-1, to=1-2]
	\arrow["{\pi}", from=1-2, to=2-1, dashed]
\end{tikzcd}
\end{equation}
where $\pi$ is the projection from $k$.

Let $\sigma_K$ be the blowup $\sigma_K\colon \widetilde{K}\to K$ of the vertex $k$ with exceptional divisor $E$. Then, the pull-back $\widetilde{Y}\coloneqq Y_1\times_K \widetilde{K}$ resolves the indeterminacy of diagram \eqref{eq_diagram_rational_maps}:
\begin{equation*}
\begin{tikzcd}
	\widetilde{Y} & \widetilde{K} & \\
	{Y_1} & {K} & {\pr 2}
	\ar[swap, "{\widetilde{\tau}}"', from=1-1, to=1-2]
	\arrow[swap, "{\phi_{2H}}", from=2-1, to=2-2]
	\arrow[swap, "{\sigma_Y}", from=1-1, to=2-1]
	\arrow[swap, "{\sigma_K}", from=1-2, to=2-2]
	\arrow["{\widetilde{\pi}}", from=1-2, to=2-3]
	\arrow[swap, "{\pi}", from=2-2, to=2-3, dashed]
\end{tikzcd}
\end{equation*}
where $\widetilde{\tau}\colon \widetilde{Y}\to \widetilde{K}$ is a degree 2 cover ramified over the divisor $E\cup \sigma_K^{-1}(D)$.

The map $\pi$ restricted to $D$ is a 3-to-1 cover of $\pr 2$, and it ramifies at a curve $C$. Throughout this section, we assume that $Y_1$ is smooth and that $C$ is irreducible and general in moduli (this is the generality assumption used in \cite{Tih82}).

A \emph{line} in $Y_1$ is defined as a smooth, purely one-dimensional subscheme $L \subset Y_1$ with Hilbert polynomial $p_H(t)=t+1$. Under our generality assumption, the scheme $F(Y_1)$ parametrizing lines in $Y_1$ is a smooth projective surface with a copy of $C$ embedded at the boundary of the locus of lines (\cite[Thm 4]{Tih82}). Indeed, $C\subset F(Y_1)$ parametrizes \emph{singular} lines (\cite[\S 3, \S 8]{Tih82}), \textit{i.e.} rational curves with a single node at a point $p \in C$, with non-reduced scheme structure at the point.


\subsection{Hyperplane sections of $Y$}\label{sec_hyperplane sections+generality}



Let $Y$ be a smooth Fano threefold of rank 1, index 2 and any degree. A smooth hyperplane section of $Y$ is a del Pezzo surface. We recall some of the related notions and some aspects of degenerations of del Pezzo surfaces. 

\begin{definition}\label{def_of_root_of_S}
Let $S$ be a hyperplane section of $Y$, \textit{i.e.} $S\in \abs{\sO_Y(1)}$. A \textit{root} of $S$ is a (possibly Weil) divisor $D$ such that $D\cdot K_S=0$ and $D^2=-2$ (see \cite[\S 8.2.3]{Dol12}).

By a root of $Y$ we mean a sheaf of the form  $\iota_*(\sO_S(D))$, where $S\in \abs{\sO_Y(1)}$, $D$ is a root of $S$, and $\iota\colon S\to Y$ is the inclusion map.
\end{definition}

Every smooth del Pezzo surface of degree $d$ arises as the blow-up of $\pr 2$ along a set $\Sigma$ containing $9-d$ points in general position \cite[Proposition~8.1.25]{Dol12}. 


We now describe hyperplane sections of $Y$:

\begin{lemma}\label{lem_hyperplane_sections}
Let $Y$ be a rank 1 index 2 Fano threefold, and let $S\in \abs{\sO_Y(1)}$ be a hyperplane section of $Y$. Suppose moreover that, if $d\leq 2$, $Y$ is general in the above sense. Then:
\begin{enumerate}
    \item $S$ is an integral, normal, Gorenstein surface with anti-canonical bundle $-K_S\simeq \sO_Y(1)_{|S}$;
    \item \label{itm2:lem_hyperplane_sections} The general $S$ is a smooth del Pezzo surface of degree $d$. If $S$ is singular, then it has ordinary double points or it is a cone over an elliptic curve. These degenerations are realized by specializing $9-d$ points on $\pr 2$ to an almost general position (in which case they become rational double points) or by degenerating a del Pezzo surface to a cone over a section of $\abs{-K_S}$;
    \item The set of roots of $S$ is finite in $\mathrm{Cl}(S)$. In any of the cases above, $S$ contains no effective roots.
\end{enumerate}
\end{lemma}

\begin{proof}
Since $Y$ has Picard rank 1, all $S$ are irreducible and reduced. Then, normality is a consequence of Zak's theorem on tangencies (\cite[Corollary~3.4.19]{positivityI}) and our generality assumptions.

All surfaces $S$ are Gorenstein, their anti-canonical bundle being $\sO_Y(1)_{|S}$ by adjunction (\cite[\S 4.1]{Kol13} applies to this case and the different vanishes \cite[Proposition~4.5(1)]{Kol13}). 

Normal Gorenstein surfaces with ample anti-canonical bundle are classified in \cite{hidaka1981}, which proves \eqref{itm2:lem_hyperplane_sections}. In particular, rational singularities are obtained by specializing $9-d$ points and contracting all $(-2)$ curves (\cite[\S 3]{hidaka1981}). 

Finally, roots are finite for a smooth del Pezzo surface $S$ of degree $d\leq 5$ \cite[\S 8.2.3]{Dol12}, and therefore also for surfaces obtained from $S$ by contracting $(-2)$ curves. If, instead, $S$ is a cone over an elliptic curve, then its divisor group contains no roots \cite[Proposition~3.14]{Kol13}. Moreover, for a root $D$ we have $D\cdot K_S=-D\cdot H_{|S}=0$ and $H_{|S}$ is ample, so $D$ cannot be effective. 
\end{proof}

\subsubsection*{Hyperplane sections of $Y_2$}
Section \ref{sec_AJY2} will require more detail on $Y_2$, we recollect here some well-known results.

A hyperplane section $\iota \colon S \hookrightarrow Y_2$ is given by the pullback $\pi^*(P)$ of a plane in $\P^3$. It is a double cover of $\pr 2$ branched over $R\cap P$, singular if and only if $P$ is tangent to $R$. Singularities are rational by Lemma \ref{lem_hyperplane_sections} (the elliptic type can only appear if $R\cap P$ is four lines meeting at a point \cite[Proposition~4.6]{hidaka1981}). 

Suppose $S$ is a smooth del Pezzo surface of degree 2: its Picard group is free of rank 8, generated by the hyperplane class $\mathbf{e}_0$ of $\pr 2$ and the exceptional curves $\mathbf{e}_i$, $i=1,\ldots,7$. The line bundle $\sO_{Y_2}(H)$ restricts on $S$ to 
\[ H_{|S}=-K_S =  3\mathbf{e}_0 - \mathbf{e_1}-\cdots-\mathbf{e_7}.\]
and the orthogonal lattice $(-K_S)^\perp$, equipped with the intersection product, is a root lattice $E_7$. The following lemma is well known.

\begin{lemma}\label{lem_hyp_sec_d=2}
Let $S$ be a hyperplane section of $Y$.
Suppose $S$ is a del Pezzo surface. Define vectors in $(K_S)^\perp$:
\begin{align*}
   \alpha_i\coloneqq 2\mathbf{e}_0-\mathbf{e}_1-\cdots-\mathbf{e}_7+\mathbf{e}_i &\quad 1\leq i \leq 7\\
   \alpha_{ij}\coloneqq \mathbf{e}_i-\mathbf{e}_j, &\quad 1\leq i<j \leq 7\\
   \alpha_{ijk} \coloneqq \mathbf{e}_0-\mathbf{e}_i-\mathbf{e}_j-\mathbf{e}_k, &\quad 1\leq i<j<k \leq 7
\end{align*}
\begin{itemize}
    \item $S$ admits 126 roots of the form $\alpha_i,\alpha_{ij},\alpha_{ijk}$. 
    \item $S$ contains 56 lines. These come in pairs $(L,-K_S-L)$ $($\textit{i.e.}, there is a unique divisor in $\abs{-K_S-L}$ and it is a line$)$.
    \item every root can be expressed $($non-uniquely$)$ as the difference of two disjoint lines. Conversely, the difference of two disjoint lines is a root in $\Pic(S)$.
\end{itemize}
If $S$ has rational singularities,
 every singular point of $S$ lies on a line. 
\end{lemma}
\begin{proof}
The roots of $S$ are computed in \cite[\S 8.2.3]{Dol12}, the lines in \cite[\S 6.1.1]{Dol12}. 

If $S$ is singular, then its minimal resolution $\tilde{S}\to S$ is obtained by the blowup of $\pr 2$ at 7 points in almost general position, and the morphism contracts all effective roots on $\tilde{S}$ \cite[Theorem 3.4]{hidaka1981}. It is straightforward to check that one can always find a line through a singular point of $S$: equivalently, every effective root of $\tilde{S}$ intersects a line.
\end{proof}

\begin{lemma}\label{lem_KodairaVan}
Let $S$ be a smooth del Pezzo surface of degree 2, and let $D$ be a root of $S$. Then, we have $H^1(D-K_S)=H^2(D-K_S)=0$.  
\end{lemma}

\begin{proof}
We apply the Kodaira vanishing theorem to the divisor $D-2K_S$, after showing that it is ample on $S$. 
The nef cone of $S$ is described in \cite[Theorem 2.4]{BC13}: one checks that if $D$ is any of the 126 roots listed in Lemma~\ref{lem_hyp_sec_d=2} then $D-2K_S$ belongs to the interior of the nef cone, and is therefore ample.
\end{proof}

\section{Construction of the moduli spaces}\label{sec_construction_of_moduli}

In this section, we fix a general Fano threefold $Y$ of Picard rank 1, index 2, and degree $d$, and describe three moduli spaces of objects of class $w$ (see \eqref{eq_def_of_v_w} for the definition of $w$) in $D^b(Y)$. Our main interest lies in the moduli space of $\sigma$-stable objects in $\Ku(Y)$, denoted $\sM_\sigma(w)$, where $\sigma$ is one of the stability conditions of Proposition \ref{thm_stab_con_on_Ku}. We will work with two intermediate moduli spaces. 

Our starting point is the moduli space $\sM_H(w)$ of sheaves of class $w$ which are stable in the sense of Gieseker with respect to the polarization $H$ (see \cite[Chapter 4]{HL10} for the definition and additional details). The space $\sM_H(w)$ is a projective scheme of finite type. We will use the fact that its Zariski tangent space at a point $[F]$ is canonically isomorphic to 
\begin{equation}\label{eq_Tangent_M_G}
    T_{[F]}\sM_H(w)\simeq \Ext^1(F,F), 
\end{equation}
and that it admits a universal family (this follows from \cite[Theorem 4.6.5]{HL10}).

Let $\sigma_{\alpha,\beta}$ be one of the weak stability conditions of Proposition \ref{def_tilt_stability}. One can define the moduli functor parametrizing $\sigma_{\alpha,\beta}$-semistable objects of fixed class $w$: this functor is corepresented by an algebraic stack of finite type \cite[Proposition~3.7]{Toda13_BG_counting}, denoted $\sM_{\alpha,\beta}(w)$, which is proper if every $\sigma_{\alpha,\beta}$-semistable objects is $\sigma_{\alpha,\beta}$-stable (quasi-properness is \cite[Theorem~1.2]{TP15}, and properness is a consequence of \cite{AP06}, with a standard argument as in \cite[Lemma 6.6]{BM14}). 

First, we describe $\sM_H(w)$ (Proposition~\ref{gieseker}) and identify it with the space $\sM_{\alpha,-\frac 12}(w)$ for $\alpha \gg 0$ (Proposition~\ref{prop_tilt}). We then show that, after crossing a wall, $\sM_{\alpha,-\frac 12}(w)$ coincides with $\sM_\sigma(w)$ (Proposition~\ref{kuzistilt}).




\subsection{The Gieseker moduli space $\sM_H(w)$}\label{sec_Gieseker_moduli}

\begin{proposition}\label{gieseker}
Let $Y$ be a Fano threefold of Picard rank 1, index 2, and degree $d$. 
The moduli $\sM_H(w)$ of Gieseker-semistable objects of class $w$ on $Y$ has two irreducible components. One, denoted $\mathcal{P}$, has dimension $d+3$ and parametrizes ideal sheaves  $I_{p|S}$ of a point in a hyperplane section $S$ of $Y$. A point $[I_{p|S}]\in \mathcal P$ is smooth in $\sM_H(w)$ iff $p$ is a smooth point of $S$.

The second component has dimension $d+1$, and its smooth points parametrize roots of $Y$ (see Definition~\ref{def_of_root_of_S}).

The components intersect in the locus parametrizing $I_{p|S}$ with $S$ singular at $p$.
\end{proposition}

\begin{proof}
Since the class $w$ is torsion, a Gieseker-stable sheaf $E$ of class $w$ must be pure. This implies that $E=\iota_*(F)$ for some sheaf $F$ supported on a hyperplane section $\iota \colon S \hookrightarrow Y$, otherwise the kernel of the map $E\to \iota_*\iota^*E$ would give a destabilizing subsheaf of smaller dimension. Stability of $E$ implies that $F$ is a torsion-free rank-one stable sheaf on $S$. 

Now we claim that $F$ must have the form $F=I_Z \otimes \sO_S(D)$, with $Z$ a scheme of finite length $z \geq 0$ and $\sO_S(D)$ a reflexive sheaf of rank 1 associated to a Weil divisor $D$ on $S$. Indeed, by Lemma \ref{lem_hyperplane_sections}, $S$ is integral and normal. Since $S$ is integral, the map $F\to F^{\vee\vee}$ is injective \cite[Theorem~2.8]{Schwede2010}. Since $S$ is normal, $F^\vee$ is reflexive and therefore $F^\vee\simeq \sO_{S}(-D)$ for some Weil divisor $D$ \cite[Section~3]{Schwede2010}. Then, $F\otimes\sO(-D)\subset F^{\vee\vee}\otimes \sO(-D)\simeq \sO_S$ and $F\otimes\sO(-D)$ must be an ideal sheaf of a subscheme supported in codimension 2. In particular, we have $\Ch(F)=\left([S],D,\Ch_2(\sO_S(D))-z[\mathrm{pt}]\right)$.

Next, we apply the Grothendieck-Riemann-Roch theorem to compute $\Ch(F)\in H^*(S)$. The computation can be done assuming that $S$ is smooth: a Riemann-Roch theorem continues to hold for local complete intersections 
\cite[\S 4]{Suw03}, and the Chern class of $N_{S/Y}\simeq \sO_Y(H)_{|H}$ does not depend on the choice of $S$ in its linear series. Then we have 
\[ \Ch(w)=\iota_*(\Ch(F) \cdot \td(N_{S/Y})) \]
which implies that $\Ch(F)=([S],D,-[\mathrm{pt}])\in H^*(S)$, where $D\in H^2(S)$ satisfies $D\cdot (H_{|S})=0$. Then, there are two possibilities: $z=1$ and $\Ch_2(F)=0$, or $z=0$ and $\Ch_2(F)=-1$. 

To proceed, we need to be more precise about the second Chern class for reflexive rank-one sheaves. The Riemann-Roch formula on normal surfaces presents a correction term $\delta_S(D)\in \Q$, only depending on $D$ and the singularities of $S$:
\[
\chi(\sO_S(D))= \chi(\sO_S) + \dfrac{1}{2}(D^2-K_S\cdot D) + \delta_S(D).
\]
For $S$ as in Lemma \ref{lem_hyperplane_sections}, $\delta_S(D)\leq 0$ (see \cite[Theorem~9.1]{Rei87} for rational singularities, and \cite[\S 7.7]{Lan00} for the elliptic ones).
Consistently with the Riemann-Roch formulae, we define the second Chern character 
\[
\Ch_2(\sO_S(D))=\dfrac{D^2}{2}+\delta_S(D).
\]
Let $p\colon \tilde{S}\to S$ denote the minimal resolution of $S$. If $z=1$, then $0=D^2=(p^*(D))^2$. Then, by the Hodge index theorem \cite[Corollary~2.16]{BHPV04},  $p^*D=0$ and therefore $D=0$ in $H^2(S)$. 
If $z=0$, then $D$ is a root unless $(p^*D)^2=\frac{D^2}{2}=-1$ and $\delta_S(D)=-\frac 12$. If $S$ is a cone, this cannot happen since $\delta_S(D)\leq -1$ \cite[\S 7.7]{Lan00}. If $S$ has ADE singularities, then $\tilde S$ is a crepant resolution, hence $K_{\tilde S}=p^*K_S=p^*(-H_{|S})$ and therefore  $(p^*D)\cdot K_{\tilde S}=D\cdot (H_{|S})=0$. But then, the Riemann Roch theorem on $\tilde S$ implies that $(p^*D)^2$ is even, a contradiction. 


For the claims about smoothness, observe the following: the locus parametrizing sheaves of the form $I_{p|S}$ has dimension $d+3=\dim \mathbb P (H^0(\sO_{Y}(1))+2$, which coincides with the rank of the Zariski tangent space of $\sM_H(w)$ at $[I_{p|S}]$ iff $p$ is a smooth point of $S$ by Lemma \ref{d+3} and the isomorphism \eqref{eq_Tangent_M_G}.

On the other hand, the locus of sheaves $\iota_*\sO_S(D)$ has dimension $d+1=\dim \mathbb P (H^0(\sO_{Y}(1))$, since $S$ has finitely many roots by Lemma \ref{lem_hyperplane_sections}. Then, each of these points is smooth in $\sM_H(w)$ by Lemma \ref{dimensiond+1} and the isomorphism \eqref{eq_Tangent_M_G}.


Now we claim that the two components intersect exactly at those points parametrizing sheaves $I_{p|S_0}$, with $S_0$ a hyperplane section of $Y$ singular at $p$. Indeed, if $S_0$ has rational singularities, then one can exhibit a flat family of sheaves of the form $\iota_*\sO_S(D)$ degenerating to $I_{p|S}$: we do this for one explicit degeneration in Example \ref{explicitfamily}, the other degenerations are completely analogous. If $S_0$ is a cone, by the properness of $\sM_H(w)$ any family of $\iota_*\sO_S(D)$ supported on a smoothing of $S_0$ has limit $I_{p|S_0}$, since $\mathrm{Cl}(S_0)$ contains no roots (Lemma \ref{lem_hyperplane_sections}). 
\end{proof}

\begin{lemma}\label{d+3}
Let $S\subset Y$ be a hyperplane section, let $p$ a smooth  point of $S$. Then we have
$$\dim \Ext^i(I_{p|S},I_{p|S})=\begin{cases}
0 & \text{ if }i=3\\
2 & \text{ if }i=2\\
d+3 & \text{ if }i=1\\
1  & \text{ if }i=0.\end{cases}
$$  
\end{lemma}

\begin{proof}
Since $I_{p/S}\simeq[\sO_S\rightarrow \C_p]$ in $D^b(Y)$, then 
$R\Hom( I_{p/S},I_{p/S})$ is the totalization of a double complex $K^{\bullet,\bullet}$ isomorphic to 
$$K^\bullet\coloneqq \left[R\Hom(\C_p,\sO_S)\rightarrow R\Hom(\sO_S,\sO_S)\oplus R\Hom(\C_p,\C_p)\rightarrow R\Hom(\sO_S,\C_p)\right],$$ so we may compute $\Ext^i( I_{p/S},I_{p/S})$ using a spectral sequence
$$E_1^{p,q}=H^q(K^{\bullet,p})\Rightarrow H^{p+q}(K^\bullet).$$
Then the first page $E^{p,q}_1$ of the spectral sequence is

\begin{center}
\begin{tabular}{c|cc}
$\vdots$ & $\vdots$ & $\vdots$\\
 $\Ext^1(\C_p,\sO_S)$ & $\Ext^1(\sO_S,\sO_S)\oplus \Ext^1(\C_p,\C_p)$ &$\Ext^1(\sO_S,\C_p)$\\
$\Hom(\C_p,\sO_S)$ & $\Hom(\sO_S,\sO_S)\oplus \Hom(\C_p,\C_p)$ &$\Hom(\sO_S,\C_p)$\\
\hline
($p=-1$)&($p=0$)&($p=1$)
\end{tabular} 
\end{center}
with arrows pointing to the right. The dimensions of the vector spaces above are given in the table below.
\begin{center}
\begin{tabular}{c|cc}
1 & 1 &0\\
1 & 3 &0\\
0 & $d+4$ &1\\
0 & 2 &1\\
\hline
\end{tabular}
\end{center}
The map in the top row is an isomorphism $\Ext^3(\C_p,\sO_S)\rightarrow \Ext^3(\C_p,\C_p)$. The map in the bottom row is  surjective because $\Hom(\sO_S,\sO_S)\rightarrow \Hom(\sO_S,\C_p)$ is.

Consider the maps in the third and second rows. We claim that if $p$ is smooth in $S$, then these two are both non-zero, and if $p$ is singular, then they are both $0$. 
For example, the map in the third row is $\Ext^1(\C_p,\C_p)\rightarrow \Ext^1(\sO_S,\C_p)$. By applying the functor $\Hom(-,\C_p)$ to the sequence $I_{p|S} \to \sO_S \to\C_p$, one sees that its kernel is $\Hom(I_{p|S},\C_p)$. On the other hand we have 
\[ \Hom(I_{p|S},\C_p)\simeq \Hom_S(I_{p|S},\C_p)\simeq \Ext^1_S(\C_p,\C_p)\simeq T_p S,\]
and the latter has rank 2, resp. 3, if $p$ is smooth, resp. singular, in $S$.
The map in the second row is $\Ext^2(\C_p,\sO_S)\to \Ext^2(\C_p,\C_p)$, and the argument for the claim is similar.

This shows that the second page of the spectral sequence is supported on the central column of the table, and hence $E_2^{p,q}$ abuts to the claimed dimensions.
\end{proof}

\begin{lemma}\label{d+4}
If $p\in S$ is singular, then 
\[
\dim \Ext^i(I_{p|S},I_{p|S})=\begin{cases}
0 & \text{ if }i=3\\
3 & \text{ if }i=2\\
d+4 & \text{ if }i=1\\
1  & \text{ if }i=0.\end{cases}
\]
\end{lemma}
\begin{proof}
From the above proof, we see that the dimensions of the objects on the second page are 
\begin{center}
\begin{tabular}{c|cc}
0 & 0 &0\\
1 & 3 &0\\
0 & $d+4$ &1\\
0 & 1 &0\\
\hline
\end{tabular}
\end{center}
and conclude that $\Ext^1(I_{p|S},I_{p|S})$ is either $\C^{d+4}$ or $\C^{d+5}$. Using the same spectral sequence as above, only this time starting with the short exact sequence 
\begin{equation}
    \label{def_IpS}
    \sO_Y(-1)\rightarrow \sI_p\rightarrow \sI_{p|S}
\end{equation}
it is not hard to see that it can only be $\C^{d+4}$.
\end{proof}

\begin{lemma}\label{dimensiond+1}
Suppose that $E=\iota_\ast F$, where $F=\sO_S(D)$ is a reflexive sheaf on a hyperplane section $\iota :S\rightarrow Y$.  Then $$R^i\sH\text{om}(E,E)=\begin{cases} \iota_\ast\sO_S & \text{ if }i=0\\ \iota_\ast\sO_S(1) & \text{ if }i=1\\
0 & \text{ otherwise.}\end{cases}$$
Therefore $\Ext^1(E,E)=\C^{d+1}$, and $\Ext^2(E,E)=\Ext^3(E,E)=0$.
\end{lemma}
\begin{proof}
By adjunction, 
\begin{equation}\label{gvduality}
   R\HHom_Y(E,E) \simeq R\HHom_S(Li^*E,F). 
\end{equation}

We can compute $E\otimes^L\sO_S$ by tensoring the short exact sequence
$\sO_Y(-1)\rightarrow \sO_Y\to \sO_S$
with $E$. Then, since $\iota^\ast=\iota^{-1} (-\otimes^L\sO_S)$, we get $L^0\iota^\ast E=F$ and $L^1 \iota^\ast E=F(-1)$. Since $F$ is reflexive, we have $\sH om_S(F,F)=\sO_S$,  so the statement follows from the spectral sequence for hypercohomology of the complex in \eqref{gvduality}.
\end{proof}

In the following example we exhibit a flat family of sheaves whose general element has the form $\iota_*\sO_S(D)$, for $\iota\colon S \to Y$ a hyperplane section and $D$ a root of $S$, which specializes to $I_{p|S_0}$ (where $p$ is a singular point in a hyperplane section $S_0$).

\begin{example}[$I_{p/S}$ as a flat limit of $\iota_*(\sO_S(D))$'s]\label{explicitfamily}
Let $f:Y\rightarrow \pr{d+1}$ be the map induced by $-K_Y/2$ (assume $d\geq 2$). Denote by $\sS$ the universal hyperplane section. Let $h_0\in (\P^{d+1})^\vee$ cut out a hyperplane section  $S_0\subset Y$ that is singular at $p$. 
Let $\Delta\subset (\P^{d+1})^\vee$ be a 1-dimensional analytic disk containing $h_0$. Hyperplane sections with rational singularities arise by specializing the $9-d$ points in $\pr 2$ (Lemma \ref{lem_hyperplane_sections}). Therefore, after possibly replacing it with a finite cover, $\Delta$ supports a family of hyperplane sections of $Y$ constructed by blowing up $9-d$ sections $s_i$ of the projection $p:\Delta\times \pr 2\rightarrow \Delta$.

For $h_t\in \Delta$, we denote by $S_t\subset Y$ the corresponding hyperplane section, and by $f_t:S_t\rightarrow \pr d$ the restriction of $f$ to~$S_t$.   The sections $s_i$ are in general position for $t\neq 0$, and they are in almost general position (see Lemma \ref{lem_hyperplane_sections}) over $t=0$. We now assume that $s_1,s_2,s_3$ are collinear over $t=0$, but one can carry over the same argument for the other degenerations. Denote the blowup with $g:\sS'\rightarrow \Delta\times \pr{2}$. The map induced by the dual of the relative dualizing sheaf $\omega_{p\circ g}$ factors through the universal hyperplane section pulled back to~$\Delta$:
\begin{center}
\begin{tikzcd}
\sS' \arrow[r, "-\omega_{p\circ g}"]  \arrow[rd,dotted]
& \Delta\times \pr {d} & \\
& \sS_\Delta \arrow[u ] \arrow[r]& \Delta\times Y 
\end{tikzcd}
\end{center}
Here the vertical map is induced by the dual of the relative dualizing sheaf of $\sS_\Delta\rightarrow \Delta$. Its fibers are $f_t:S_t\rightarrow \pr d$.

Let $\sE_i$ be the exceptional divisors of $g$. Define $D=-g^\ast(\sO_{\pr{2}}(1))+\sE_1+\sE_2+\sE_3$. Then the fiber of $D$ over $t\neq 0$ is a Cartier divisor $D_t$ of degree $0$ and $D_t^2=-2$, and when $t=0$, $D$ restricts to $-E$, where $E$ is the $(-2)$ curve in the central fiber of the blowup $g$. If we pushforward to $\Delta\times Y$, we get a flat family of divisors $\iota_*(\sO_{S_t}(D_t))$ over a general fiber, and $I_{p/S_0}$ in the central fiber.
\end{example}

\subsection{Moduli spaces and wall crossing in the $(\alpha,\beta)$-plane}\label{sec_tilt moduli}

Now we turn to the study of the wall crossing; we investigate moduli spaces of stable objects of class $w$ with respect to a stability condition $\sigma_{\alpha,\beta}$ of Proposition \ref{def_tilt_stability}. In particular, we fix $\beta=-1/2$, $\alpha^+ \gg 1$ and $0<\alpha^-\ll 1$. We denote the corresponding moduli spaces by $\sM_+(w)$, $\sM_-(w)$ respectively. In Proposition \ref{prop_tilt} we describe them and show that the former coincides with $\sM_H(w)$ and is separated from the latter by a single wall. We need a preparatory Lemma:

\begin{lemma}\label{possiblesubs}
The only tilt-semistable sheaves  $A\in \Coh_{\alpha,-\frac{1}{2}}(Y)$ having truncated Chern class 
$\Ch^{-\frac{1}{2}}_{\leq 2}(A)=\left(1,\dfrac{H}{2},\dfrac{H^2}{8}\right)$ 
are $\sO_Y$ and $I_Z$, with $Z$ a zero-dimensional subscheme.
\end{lemma}
\begin{proof}
Such an object $A$ has $H^2\Ch_1^{-\frac{1}{2}}(A)=\frac{H^3}{2}$, which is the minimal positive value of $H^2\Ch_1^{-\frac{1}{2}}$ on objects of $\Coh_{\alpha,-\frac{1}{2}}(Y)$. By \cite[Lemma 7.2.2]{BMT14}, $A$ must be a slope-stable sheaf, so in particular torsion-free. Since the untwisted Chern character of $A$ is 
$\Ch_{\leq 2}(A)=\left(1,0,0\right)$,
we must have $A\simeq \sO_Y$ or $A\simeq I_Z$ for some zero-dimensional subscheme $Z$.
\end{proof}

We will also need the following definition: let $p$ be a point on $Y$ and $I_p$ its ideal sheaf; applying the functor $\Hom(\sO_Y, -)$ to the short exact sequence
\[
0 \to I_p(-1)\to \sO_Y(-1) \to \C_p \to 0
\]
and since $\sO_Y(-1)$ has no cohomology, one computes that $\Ext^1(\sO_Y,I_p(-1))=\C$. Serre duality with $K_Y=\sO_Y(-2)$ gives $\Ext^1(\sO_Y,I_p(-1))=\Ext^1(I_p, \sO_Y(-1)[1])^\vee$: let then $E_p$ be the unique non-trivial extension \eqref{eq_def_of_Ep} so that $[E_p]=w$.

\begin{proposition}\label{prop_tilt}
There are only two chambers in the $(\alpha,\beta)$-plane for the class $w$, and the line $\beta=-\frac 12$ intersects both. 
The moduli space of $\sigma_{\alpha,-\frac 12}$-semistable objects for $\alpha\gg 1$ is isomorphic to $\sM_H(w)$.
If $0< \alpha \ll 1$, $\sigma_{\alpha,-\frac 12}$-semistable objects of class $w$ have the form $\iota_*\sO_S(D)$ (where, as usual, $D$ is a root of a hyperplane section $S$), or $E_p$ as in \eqref{eq_def_of_Ep}.
\end{proposition}
\begin{proof}
The isomorphism $\sM_+(w)\simeq \sM_H(w)$ is a standard argument: every object in $\sM_+(w)$ is Gieseker-stable \cite[Lemma 7.2.1]{BMT14}. Conversely, assume $E$ is a Gieseker-stable sheaf of class $w$ on $Y$. Then $H^2\Ch^{-\frac 12}_1(E)=H^3$ is twice the minimal positive value of $H^2\Ch^{-\frac 12}_1$ on $\Coh^{-\frac12}(Y)$. Then, the only possibility for a destabilizing sequence $A \to E\to B$ is that $A$ satisfies $H^{2}\Ch^{-\frac12}_1(A)\leq\frac{H^3}{2}$, and hence $A$ must be a slope stable sheaf of dimension $\geq 2$ (by applying \cite[Lemma 7.2.2]{BMT14} to $A$, and observing that $H^{-1}(A)=0$, since $H^{-1}(E)=0$). If $A$ has dimension 2, we must have $H^{-1}(B)=0$, but $E$ is pure of dimension 2 so $A$ cannot destabilize it. If, instead, $\rk(A)>0$, we must have $\mu_H(A)>-\frac{1}{2}$ and therefore $H^2\Ch^{\frac 12}_1(A)>0$. Then, one has 
\[ \mu_{\alpha,-\frac12}(A)= \frac{H\Ch^{\frac12}_2(A) - \frac{\alpha^2}{2}\Ch_0(A)}{H^2\Ch^{\frac 12}_1(A)} \xrightarrow{\alpha \to +\infty} - \infty,\]
so $A$ cannot destabilize for $\alpha$ sufficiently large. Equivalently, $E$ is $\sigma_{\alpha^+,-\frac12}$-stable if and only if it is Gieseker stable, and the two moduli spaces are isomorphic.

Since $w$ is a torsion class, walls in the $(\alpha,\beta)$-plane are concentric semicircles centered at $\alpha=0$ and $\beta = -\dfrac{1}{2}$ \cite[Theorem 2.2]{BMSZ17}. In particular, the line $\beta = -\dfrac{1}{2}$ intersects all the possible walls for the class $w$. The twisted Chern character of $E$ satisfies
\[
\Ch^{-\frac{1}{2}}_{\leq 2}(E) = (0,H,0),
\]
hence if $A$ gives a numerical wall for $E$ and $\Ch^{-\frac{1}{2}}_{\leq 2}(A)=(x,y,z)$, the equation of the wall at the point with coordinate $\beta=-\dfrac{1}{2}$ yields
\[
z-\dfrac{\alpha^2}{2}x=0
\]
meaning that $z$ and $x$ must have the same sign and cannot be equal to 0 (otherwise they would not give a wall).

An actual wall in tilt-stability intersecting the vertical line $\beta=-1/2$ must be given by a short exact sequence $0\to A \to E \to B \to 0$ with $0< \Ch_1^{-1/2}(A)< \Ch_1^{-1/2}(E)=1$; since $\Ch_1^{-1/2}(A)\in \frac{1}{2}\Z$, we get $\Ch_1^{-1/2}(A)=\dfrac{1}{2}$.

Moreover, by the support property one has $0 \leq \Delta(A) \leq \Delta(E)=1$, which can be rearranged to
\[
-1 \leq -8xz \leq 3;
\]
since $\beta=-\dfrac{1}{2}$, it follows that  $8z \in \Z$, and since $x$ and $z$ have the same sign and cannot be 0 we must have $x=1$ and $z=\dfrac{1}{8}$.

This means there is only one possible actual wall, and it would be given by a subobject with $\Ch^{-\frac{1}{2}}_{\leq 2}(A)=\left(1,\dfrac{H}{2},\dfrac{H^2}{8}\right)$. As before, $A$ must be a sheaf and we can assume $A$ to be tilt-semistable. By Lemma \ref{possiblesubs}, either $A\simeq \sO_Y$ or $A\simeq I_Z$ for $Z$ a zero-dimensional subscheme.

Now if $E=\iota_*\sO_S(D)$, with $D$ as in Proposition~\ref{gieseker}, neither of the $A$'s can map to $E$: in both cases, a map $A\to E$ would produce a section of $\sO_S(D)$, by restricting to $S$ and then dualizing $A_{|S} \to \sO_S(D)$ twice. This does not happen since $D$ is not effective by Lemma \ref{lem_hyperplane_sections}.

On the other hand, there are always maps $I_p \to I_{p/S}$ fitting in the triangle
\[
0\to I_p \to I_{p/S} \to \sO(-1)[1]\to 0.
\]
This means that the numerical wall is an actual wall: all objects in $\sM_H(w) \setminus \mathcal{P}_d$ are stable on both sides of the wall, while the $I_{p/S}$'s are destabilized. In turn, the unique extensions
\[
0 \to \sO(-1)[1]\to E_p \to I_p \to 0,
\]
are also unstable at the wall and become stable below it, since both $\sO(-1)[1]$ and $I_p$ are stable, and stay such in the whole chamber because there are no other walls.
\end{proof}

\begin{lemma}\label{lem_Families_IpEp}
There is an embedding $\phi\colon Y\to \sM_-(w)$, whose image $\mathcal Y$ parametrizes objects defined as in \eqref{eq_def_of_Ep}.
\end{lemma}

\begin{proof}
Let $\sI\in D^b(Y\times Y)$ be a universal family of ideal sheaves of points of $Y$. Then, using the notation of \ref{sec_basechange_mutations}, we consider the projection $\mathbb R_{\sO_Y(-1)}^Y(\sI) \in \!^\perp\pair{\sO_Y(-1)}_Y$. The fibers of $\mathbb R_{\sO_Y(-1)}^Y(\sI)$ over $Y\times \set{[I_p]}$ fit in triangles
\[ \mathbb R_{\sO_Y(-1)}^Y(\sI)_p \to I_p \to \sO_Y(-1)[2]. \]
Hence $\mathbb R_{\sO_Y(-1)}^Y(\sI)_p\simeq E_p$ by \eqref{eq_def_of_Ep}. Therefore, $\mathbb R_{\sO_Y(-1)}^Y(\sI)$ is a family of complexes in $\sM_-(w)$ of the form \eqref{eq_def_of_Ep} over $Y$, which induces a morphism $\phi\colon Y \to \sM_-(w)$. The morphism is clearly injective at the level of sets, and it is an embedding because it is injective on tangent spaces, as shown in Lemma \ref{lem_numerics}.
\end{proof}

\begin{lemma} \label{lem_numerics}
Fix $p\in Y$ and let $E_p$ be one of the objects as in (\ref{eq_def_of_Ep}); then there is an identification
 \begin{equation}\label{eq_decomposition_of_tangent_space}
     \Ext^1(E_p,E_p)= \Ext^1(I_p,I_p) \oplus V_p\simeq \C^3\oplus V_p,
\end{equation}
where $V_p$ is the space of hyperplanes tangent to $Y$ at $p$.
\end{lemma}

\begin{proof}
To describe the splitting, first observe that $\Ext^1(E_p,E_p)\simeq \Ext^1(E_p,I_p)$ (\textit{e.g.} by applying $\Hom(E_p,-)$ to the sequence \eqref{eq_def_of_Ep}). Applying $\Hom(-,I_p)$ to \eqref{eq_def_of_Ep} one obtains
\begin{equation*}
\begin{split}
    &\Ext^3(E_p,E_p)=0; \\
     0\to \Ext^1(I_p,I_p) \to   \Ext^1(E_p,I_p) & \to \ker\left(\Ext^1(\sO(-1)[1],I_p)\xrightarrow{\alpha} \Ext^2(I_p,I_p)\right) \to 0,
\end{split}
\end{equation*}
where $\alpha$ is pre-composition with the unique map $\epsilon_p :I_p\to \sO(-1)[2]$ of the triangle \eqref{eq_def_of_Ep}. Let $V_p\coloneqq \ker\alpha$. Since $\Ext^1(\sO(-1)[1],I_p)\simeq H^0(I_p(1))$, every element $f\in \Ext^1(\sO(-1)[1],I_p)$ can be identified with a hyperplane section $S_f$ containing $p$.

We now prove that a nonzero $f\in \Ext^1(\sO(-1)[1],I_p)$ is in $\ker \alpha$ if and only if $S_f$ is singular at $p$. We have the following diagram:
\begin{center}
\begin{tikzcd}[]
  & I_p \dar["\epsilon_p"] &\\
I_{p/S}[1] \rar &\sO(-1)[2] \rar["f"] & I_p[2].
\end{tikzcd}
\end{center}
Then $\alpha(f)=f\epsilon_p=0$ if and only if $\epsilon_p$ lifts to a map $I_p\rightarrow I_{p/S}[1]$. Next, we check that $\epsilon_p$ is in the image of $\Hom(I_p,I_{p/S}[1])\rightarrow \Hom(I_p,\sO(-1)[2])\simeq \C$ if and only if $p\in S$ is a singular point. Indeed, applying $\Hom(I_p,-)$ to the triangle $$\sO(-1)\to I_p\to I_{p|S}$$ we get a long exact sequence
$$0\to \C^3\to \Ext^1(I_p,I_{p|S})\to \Hom(I_p,\sO(-1)[2]) \to \cdots$$
Thus, it suffices to argue that $\Ext^1(I_p,I_{p|S})$ is $3$-dimensional if $p\in S$ is smooth, and $4$-dimensional if $p\in S$ is singular. To see this, apply $\Hom(-,I_{p|S})$ to the sequence $I_p\to \sO_Y\to \C_p$.
\end{proof}

\begin{remark}\label{rmk_DimOfVp}
A simple dimension count shows that $V_p$ has dimension $d-2$ for degree $d\geq 3$. For degrees $d=2$ (resp. $d=1$), $V_p$ is empty unless $p\in R$ (resp. $p\in C$), in which case $V_p$ is one-dimensional, generated by the unique tangent hyperplane to $p$. 
\end{remark}

\begin{proposition}\label{prop_GiesekerVSTilt}
Let $\sY =\phi(Y)$ the locus defined in Lemma \ref{lem_Families_IpEp}.
Wall-crossing induces a surjective morphism $s:\sM_H(w) \to \mathcal{M}_-(w)$. Moreover, $s$ is an isomorphism on $\sM_H(w) \setminus \mathcal{P}\simeq \mathcal{M}_-(w)\setminus \sY$ and its fiber over a point $p$ of $\sY$ is the projective space $\mathbb P(H^0(I_p(1)))$. 
\end{proposition}

\begin{proof}

To define $s$, we construct a family of $\sigma_{\alpha,\beta}$-stable objects on $\sM_H(w)$, starting from the universal family $\sG$ of Gieseker stable sheaves on $Y\times \sM_H(w)$. 

Let $\mathcal J \coloneqq \mathbb R_{\sO_Y(-1)}^{\sM_H(w)}(\sG)\in D(Y\times \sM_H(w))$ be the projection of $\sG$ onto the category $\!^\perp\pair{\sO_Y(-1)}_{\sM_H(w)}$ (notation as in Section \ref{sec_basechange_mutations}). We claim that $\mathcal J$ parametrizes $\sigma_{\alpha,\beta}$-stable object, and therefore induces a morphism $s\colon\sM_H(w)\to \sM_-(w)$. 

This can be checked after restricting to fibers $Y\times \set{x}$ for $x\in \sM_H(w)$. 
If $x$ represents a sheaf $I_{p|S}$, then one computes $\Hom^\bullet(I_{p|S},\sO_Y(-1))\simeq \C[1]\oplus \C[2]$ and sees that
$\mathcal J_x$ fits in the right mutation triangle
\[ \mathcal J_x \to I_{p|S} \xrightarrow{\mathrm{coev}} \sO_Y(-1)[1]\oplus \sO_Y(-1)[2]. \]
Applying the octahedral axiom to the composition of $\mathrm{coev}$ with the projection onto the summand $\sO_Y(-1)[1]$, we find a distinguished triangle 
\[ \mathcal J_x \to I_p \to \sO_Y(-1)[2]. \]
It follows that $\mathcal J_x$ is isomorphic to $E_p$ (see the triangle \eqref{eq_def_of_Ep}) and is $\sigma_{\alpha,\beta}$-stable. 

On the other hand, if $x$ represents a sheaf $\iota_*\sO_S(D)$, we have 
\[ \Hom^*(\iota_*\sO_S(D),\sO_Y(-1)) \simeq \Hom^{3-*}(\sO_Y(1), \iota_*\sO_S(D))^\vee=0, \]
by Serre duality and Lemma \ref{osd} below, so $\mathcal J_x\simeq \iota_*\sO_S(D)$ is $\sigma_{\alpha,\beta}$-stable.

This defines the map $s$. The statements about its properties also follow from the description of the objects $\mathcal J_x$ just given: for example, the map $s_{|\mathcal P}\colon \mathcal P \to \mathcal Y$ coincides with the assignment $[I_{p|S}]\mapsto p$, and the fiber over a point $p\in \mathcal Y$ is the projective space $\mathbb P(\Hom(\sO_Y(-1),I_p))\simeq\mathbb P(H^0(I_p(1)))\simeq \pr d$.
\end{proof}

\subsection{The moduli space $\sM_\sigma(w)$}\label{sec_sigma_moduli}

Fix $0<\alpha\ll 1$ and $\beta=-\frac 12$. Consider the weak stability condition $\sigma^0_{\alpha,\beta}=(\Coh^0_{\alpha,\beta}(Y), Z_{\alpha, \beta}^0)$ on $D(Y)$ defined in Proposition \ref{tiltoftilt}. Recall  that $\Coh^0_{\alpha,\beta}(Y)$ is a tilt of $\Coh^\beta(Y)$, and 
\[
Z_{\alpha,\beta}^0(E) = H^2\Ch^{\beta}_1(E) +i \left(H\Ch_2^\beta(E) -\dfrac{\alpha^2}{2}H^3\Ch_0(E)\right).
\]
The condition $\sigma^0_{\alpha,\beta}$ gives rise to a stability condition $\sigma=(\Ku(Y) \cap \Coh(X)^{0}_{\alpha, \beta},Z^0_{\alpha,-\frac{1}{2}})$ on $\Ku(Y)$ by Proposition~\ref{thm_stab_con_on_Ku}. In this section, we focus on the moduli space $\mathcal{M}_\sigma(w)$ of $\sigma$-semistable objects of $\Ku(Y)$ of class $w$, and we prove:

\begin{proposition}\label{kuzistilt}
$\mathcal{M}_\sigma(w)$ is isomorphic to $\mathcal{M}_-(w)$.
\end{proposition}

We start by relating stability with respect to $\sigma$,  $\sigma^0_{\alpha,-\frac12}$, and $\sigma_{\alpha,\beta}$.

\begin{remark} \label{rmkstable} Notice that  with this choice of stability condition, $Z^0_{\alpha,\beta}(w)=1$, so  an object $E$ of class $-w$ is stable if and only if an even shift $E[i]$ of it belongs to the heart $\Coh^0_{\alpha,\beta}(Y)$. Indeed, such an object $E[i]$ has infinite slope, so it is semistable; moreover, the class $w$ is primitive in $\Ku(Y)$, so that there are no strictly semistable objects of that class.
\end{remark}


\begin{lemma}\label{connection}
Let $E\in \Ku(Y)$ of class $-w$. Then $E$ is stable with respect to $\sigma$ if and only if $E$ is $\sigma^0_{\alpha,-\frac{1}{2}}$-semistable in $D(Y)$. 
\end{lemma}
\begin{proof}
If $E \in \Ku(Y)$ is $\sigma$-stable, up to a shift we have $E \in \mathcal{A} = \Coh^0_{\alpha,-\frac{1}{2}}(Y) \cap \Ku(Y)$, so in particular $E\in \Coh^0_{\alpha,-\frac{1}{2}}(Y)$. Since $E$ is in the heart of a (weak) stability condition and it has maximal phase (because $\im Z^0_{\alpha,\beta}(E)=\im Z^0_{\alpha,\beta}(w)=0$), by definition $E$ must be semistable.

Conversely, if $E\in \Ku(Y)$ is semistable in $\Coh^0_{\alpha,-\frac{1}{2}}(Y)$ then in particular $E$ belongs to the heart $\mathcal{A}$, hence it is stable by the above Remark \ref{rmkstable}.
\end{proof}


\begin{lemma}\label{prop_semistable_objects2}
If $E\in \Coh^0_{\alpha,\beta}(Y)$ be $\sigma^0_{\alpha,\beta}$-semistable object of class $-w$, then $E[-1]\in \Coh^\beta(Y)$ and it is $\sigma_{\alpha,\beta}$-semistable.
\end{lemma}
\begin{proof}
Since $E$ is in $\Coh^0_{\alpha,\beta}(Y)$, there is a triangle
$$E'[1]\to E\to T$$
with $E'\in \Coh^\beta(Y)_{\mu_{\alpha,\beta}\leq 0}$, $T\in \Coh^\beta(Y)_{\mu_{\alpha,\beta}> 0}$. Since $E$ is semistable with respect to $\mu^0_{\alpha,\beta}$, $Z_{\alpha,\beta}(T)$ has to be $0$, so $T$ is supported on points, that is, $T$ has finite length. 
Also, $E'$ needs to be $\sigma_{\alpha,\beta}$-semistable, because otherwise the destabilizing subobject of $E'$ would make $E$ $\sigma_{\alpha,\beta}^0$-unstable.

We now prove that $T$ must vanish, and conclude.
By a result of Li \cite{Li15}, we have $\Ch^\beta_3(E')\leq 0$. Now, $\Ch(E')=w + t\frac{H^3}{d}$, where $t$ denotes the length of $T$. 
The inequality then reads
\[ \Ch^\beta_3(E')=H^3\left[ \frac{t}{d} + \frac{(d-6)}{6d} -\frac18\right]=H^3\left[ \frac{1}{24} + \frac{t-1}{d} \right]\leq 0 \]
which shows $t=0$, and therefore $T=0$.
\end{proof}

Lemma \ref{connection} and Lemma \ref{prop_semistable_objects2}, together, show that if an object $E\in \Ku(Y)$ is $\sigma$-semistable then it is $\sigma_{\alpha,\beta}$-semistable. Next, we work towards a converse: we classified all $\sigma_{\alpha,\beta}$-semistable objects of class $w$ in Proposition~\ref{prop_tilt}, we check that they belong to $\Ku(Y)$ in Lemma \ref{ep} and Lemma \ref{osd} below.



\begin{lemma} \label{ep}
Let $E_p$ be defined as in \eqref{eq_def_of_Ep}. Then $\Hom^\bullet(\sO_Y(1),E_p)=\Hom^\bullet(\sO_Y,E_p)=~0$. In other words,
the objects $E_p$ belong to $\Ku(Y)$.
\end{lemma}
\begin{proof}From the short exact sequence
\[
0 \to I_p\to \sO_Y \to \C_p \to 0,
\]
it is straightforward to see that $\Hom^\bullet(\sO_Y,I_p)=0$, and we already pointed out that $\sO_Y(-1)$ has no cohomology, so that $\Hom^\bullet(\sO_Y,E_p)=0$ using the defining sequence \eqref{eq_def_of_Ep}.

On the other hand, from the sequence above we have $\Hom^\bullet(\sO(1),I_p)\simeq \C[2]$, and again by applying $\Hom(\sO_Y(1),-)$ to the distinguished triangle \eqref{eq_def_of_Ep}  we see that $\Hom^\bullet(\sO_Y(1),E_p)=0$.
\end{proof}

\begin{lemma} \label{osd}
Let $S$ be a hyperplane section of $Y$, and $D$ a root of $S$. Then $$\Hom^\bullet(\sO_Y(1), \iota_*(\sO_S(D))=\Hom^\bullet(\sO_Y, \iota_*(\sO_S(D))=0.$$ In other words, $\iota_*\sO_S(D)\in \Ku(Y)$.
\end{lemma}
\begin{proof}
By adjunction of the functors $\iota_*$ and $\iota^*$, it's enough to check the required vanishing on $S$.
Let $H$ be the pullback of the ample divisor on $Y$ to $S$: we get $\Hom^\bullet(\sO_Y(1), \iota_*(\sO_S(D))\simeq H^\bullet(S,\sO_S(D-H))$ and $\Hom^\bullet(\sO_Y, \iota_*(\sO_S(D))\simeq H^\bullet(S,\sO_S(D))$.
The adjunction isomorphisms above, and the condition $\chi([\sO_Y],w)=\chi([\sO_Y(1)],w)=0$, imply that $\chi(S,\sO_S(D))=\chi(S,\sO_S(D-H))=0$. Since $D$ and $D-H$ are not effective (see Lemma \ref{lem_hyperplane_sections}), we have $H^0(D)=H^0(D-H)=0$, and by Serre duality it follows that $H^2(D)=H^2(D-H)=0$. Then $H^1(D)=H^1(D-H)=0$ must be $0$ as well, and the claim is proven.
\end{proof}

We are now ready to prove Proposition \ref{kuzistilt} and lastly Theorem \ref{introduction_1}:

\begin{proof}[Proof of Proposition \ref{kuzistilt}]
The proposition follows from the fact that the two moduli spaces $\sM_-(w)$ and $\sM_\sigma(w)$ corepresent isomorphic functors. Observe that stability and semistability coincide for both $\sigma$ and $\sigma_{\alpha,\beta}$ and class $w$. 

Then, for all $E\in D(Y)$, we have that $E$ is a $\sigma$-stable object of $\sA$ if and only if $E[-1]\in\Coh^\beta(Y)$ and $E$ is $\sigma_{\alpha,\beta}$-stable: the forward direction follows from Lemma \ref{connection} and Lemma \ref{prop_semistable_objects2}. For the converse, observe that $\sigma_{\alpha,\beta}$-stable objects have been classified in Proposition \ref{prop_tilt}, and all belong to $\sA$ (after a shift) by Lemma \ref{ep} and Lemma \ref{osd}. Therefore, they must be stable by Remark \ref{rmkstable}.
\end{proof}

We recollect here the results of Section \ref{sec_construction_of_moduli}:

\begin{thm}\label{thm_Section2}
The moduli space $\sM_\sigma(w)$ is proper. There is a subvariety $\sY \subset \sM_\sigma(w)$, isomorphic to $Y$, parametrizing objects $E_p$ fitting in a distinguished triangle \eqref{eq_def_of_Ep}. The complement $\sM_\sigma(w)\setminus \sY$ is smooth of dimension $d+1$ and parametrizes roots of $Y$.
\end{thm}

\begin{proof}
The statement is a combination of the description of $\sM_-(w)$ (Proposition \ref{prop_tilt}, Proposition \ref{prop_GiesekerVSTilt}), and the isomorphism $\sM_\sigma(w)\simeq \sM_-(w)$ (Proposition \ref{kuzistilt}). Properness follows from properness of $\sM_-(w)$ (see the beginning of this section).
\end{proof}


\subsection{The Abel--Jacobi map on $\sM_\sigma(w)$}
\label{sec_AJ_map}
Here we recall some facts about the intermediate Jacobian $J(Y)$ of a Fano threefold $Y$ and its associated Abel--Jacobi map (we refer the reader to \cite[Section~12.1]{Voi03_volumeI} for more details), with particular focus to the case of curves on threefolds. We also recall the \textit{tangent bundle sequence (TBS)} technique developed in \cite[Section~2]{Wel81} to study the infinitesimal behavior of the AJ map. 

The \emph{intermediate Jacobian} of $Y$ is the complex torus 
\[
J(Y) \coloneqq \bigslant{H^1(Y, \Omega^2_Y)^{\vee}}{H^3(Y,\Z)}
\]
and it comes equipped with an \emph{Abel--Jacobi} map $\Phi \colon\mathcal{Z}_1(Y)^{\text{hom}} \longrightarrow  J(Y)$, where $\mathcal{Z}_1(Y)^{\text{hom}}$ is the space of 1-dimensional algebraic cycles on $Y$ which are homologous to 0, defined as follows. Every $Z\in \mathcal{Z}_1(Y)^{\text{hom}}$ admits a 3-chain $\Gamma$ in $Y$ such that $\partial \Gamma=Z$. Then one has the linear map on $H^{3,0}(Y)\oplus H^{2,1}(Y)$ given by integration against $\Gamma$, whose image in $J(Y)$ (projecting onto the second factor) is well-defined and defines the map 
\[ \Phi(Z) \coloneqq  \int_\Gamma. \]

Now suppose $T$ is a smooth variety parametrizing algebraic 1-cycles $\{Z_t\}_{t \in T}$ on $Y$. Picking a base point $0\in T$ yields a map $T\to \mathcal{Z}_1(Y)^{\text{hom}}$, which one composes with $\Phi$ to define an analogous map (see \cite[Section~2]{Wel81}):
\[
\Psi \colon \xymatrix@R-2pc{T \ar[r] & J(Y) \\ Z_t \ar@{|->}[r] & \Phi(Z_t-Z_0).}
\]

We will apply this construction to families of complexes of sheaves, using the (Poincar\'e dual of) the second Chern class $c_2$: to every point in $\sM_\sigma(w)$, parametrizing an object $F\in \Ku(Y)$, we associate an algebraic 1-cycle $c_2(F)$. Abusing notation, we identify points of $\sM_\sigma(w)$ with the objects they parametrize, and pick a base point $F_0$. Then, we get a map from $\sM_\sigma(w)$ to the intermediate Jacobian:
\[
\Psi \colon \xymatrix@R-2pc{\mathcal{M}_\sigma(w) \ar[r] & J(Y) \\ F \ar@{|->}[r] & \Phi(c_2(F)-c_2(F_0)).}
\]

Finally, we recall the TBS technique of \cite[Section~2]{Wel81} for the infinitesimal study of $\Psi$. Suppose $Y$ is embedded in a quasi-projective fourfold $W$. Let $Z$ be a smooth curve on $Y$, and $0\in T$ a smooth family of deformations of $Z=Z_0$, so that $\Psi$ is defined on $T$. Then the dual map of the differential of $\Psi$ at $0$, denoted $\psi_Z^\vee$, appears in the following diagram:
\begin{equation}\label{eq_diagram_AJ_codiff}
\begin{tikzcd}[]
H^0(N_{Y/W}\otimes K_Y) \rar["R"] \ar["\mathrm{res}"]{d} & H^1(\Omega^2_Y) \rar \ar["\psi_Z^\vee"]{d} & 0\\
H^0((N_{Y/W}\otimes K_Y)_{|Z}) \ar["\beta_Z"]{r} & H^1(N_{Z/Y} \otimes K_Y) \rar & H^1(N_{Z/W} \otimes K_Y).
\end{tikzcd}
\end{equation}
Here, $\mathrm{res}$ is the restriction map, and the rows are part of long exact sequences associated, respectively, with the short exact sequence
\[ \Omega_Y^2 \otimes N_{Y/W}^\vee \to \Omega_W^3 \otimes \sO_Y \to \Omega_Y^3, \]
(\cite[2.6]{Wel81}) and with the normal bundle sequence of $Z\subset Y \subset W$. We use the isomorphism $H^1(N_Z Y \otimes K_Y)\simeq H^0(N_Z Y)^\vee$ and the Kodaira-Spencer isomorphism $H^0(N_Z Y)^\vee \simeq \Omega_{T,0}$.

\begin{remark}
\label{rmk_YContractedByAJ}
The map $\Psi$ contracts the locus $\sY\subset \sM_\sigma(w)$, for every $d$. One can show this in two different ways: on the one hand, the objects $E_p$ all have the same second Chern class, so $\Psi$ is constant on the set $\sY$ parametrizing them. Alternatively, the Abel--Jacobi map contracts rational curves since it maps to an abelian variety, and $Y$ (and hence $\sY$) is rationally connected for all degrees.
\end{remark}

\section{Case-by-case description of $\sM_\sigma(w)$}\label{sec_higher_degrees}
\label{sec_d=2}

In this section we illustrate Theorem \ref{thm_Section2} on a case-by-case basis. For degrees $d=3,4,5$ we recover known descriptions of $\sM_\sigma(w)$ and of the Abel--Jacobi map. We follow the notation given in Section \ref{sec_Fano_threefolds}.

\subsection{Degree 5}


We give the alternative description of $V_5$ from \cite{Fae05} and recall the properties of the quiver moduli space from \cite{Dre88}.

Consider a three-dimensional vector space $B$ and the projective spaces $\pr 2= \mathbb P(B)$ and $\Check{\mathbb{P}}^2=\mathbb{P}(B^\vee)$. Fix a smooth conic $F\in \Sym^2 (B)$.
Then, viewing $F$ as a function $F\colon \Sym^2(B^\vee)\to \C$, one sees as in \cite[Lemma 2.3]{Fae05} that the variety $V_5$ is isomorphic to
\begin{equation}\label{eq_locus_in_hilb}
     \left\lbrace Z\in \Hilb_3(\Check{\mathbb{P}}^2) \mid H^0(I_Z(2))\subset \ker F\right\rbrace \subset \Hilb_3(\Check{\mathbb{P}}^2), 
\end{equation} 
and that there is a natural isomorphism $B\simeq \Hom(S,Q^\vee)$.

On the other hand, the universal bundles 
 have classes
\[ [S]=2-H +\frac{H^2}{10} +\frac{H^3}{30} \qquad\qquad [Q^\vee]=3-H-\frac{H^2}{10}+\frac{H^3}{30} \]
in $K_{num}(\Ku(Y))$ (see \cite[Section~3]{Fae05}), so that $w=2[Q^\vee]-3[S]$. In other words, $\sM_\sigma(w)$ is identified with the moduli space of quiver representations of the form $S^{\oplus 3}\rightarrow (Q^{\vee})^{\oplus 2}$.

In turn, such a map is a point of the variety $\mathbb G (2\times 3,B^\vee)$ of $2\times 3$ matrices on the space $B^\vee$, \textit{i.e.}
\[\sM_\sigma(w) \simeq \mathbb G (2\times 3,B^\vee)= \mathbb P(M_{2,3}(B^\vee))// \SL(2)\times \SL(3).\]

The variety $\mathbb{G}(2\times 3,B^\vee)$ has been studied in detail in \cite{Dre88}. The morphism 
\begin{align*}
 \Hilb_3(\Check{\mathbb{P}}^2)  &\longrightarrow  \Gr(3,6)= \Gr(3,\Sym^2(B))\\
      \phantom{aaa} Z \phantom{aaa} &\longmapsto  H^0(I_Z(2))
\end{align*}
factors through
\[ q\colon \Hilb_3(\Check{\mathbb{P}}^2) \to \mathbb G (2\times 3,B^\vee), \]
defined by sending a subscheme $Z\subset \Check{\mathbb P}^2$ to the map between the first order syzygies and the ideal generators of $I_Z(2)$. The variety $\mathbb G (2\times 3,B^\vee)$ is a birational model of $\Hilb_3(\Check{\mathbb{P}}^2)$, it compactifies the locus of reduced subschemes, and $q$ contracts the locus of collinear schemes to a plane in $\mathbb{G}(2\times 3, B^\vee)$ \cite[Theorem~4]{Dre88}. Moreover, $q$ is an isomorphism on the set \eqref{eq_locus_in_hilb} \cite[Lemma 2.3]{Fae05}. 

In summary, $V_5$ can be described in $\Gr(3,\Sym^2(B))$ as the intersection of the moduli space $\sM_\sigma(w)=\mathbb{G}(2\times 3, B^\vee)$ with the cell $\Gr(2,\Sym^2(B)/\pair{F})$ of subspaces containing $F$ (since this condition is the one appearing in \eqref{eq_locus_in_hilb}):
\[V_5=\sM_\sigma(w)\cap \Gr(2,5)\subset \Gr(3,6).\]

The intermediate Jacobian of $V_5$ is a point, so there is not much interesting to say about the Abel--Jacobi map \cite[Table 12.2]{AGV_Fano}. 

\subsection{Degree 4}

By construction, the functor $\phi$ sends points of $C$ to (twists of) spinor bundles on $X$. The numerical class in $K_{num}(Y)$ of both spinor bundles is $[S_\pm]=2+H-\frac{H^3}{12}$ \cite[Remark 2.9]{Ott88}. This implies \[[S(-1)]=2v-w.\]

On the other hand, ideal sheaves of lines of $Y$ are images of line bundles on $C$: suppose in fact that $\phi(M)=I_\ell$ for some $M\in \Ku(Y)$ and some line $\ell\subset Y$. Then we have $\Hom(M,M)\simeq \Hom(I_\ell,I_\ell)=\C$, and by Riemann-Roch
\[ -1 = \chi_Y(v,v)=\chi_C(M,M)=-\rk(M)^2, \]
therefore $M$ is a line bundle. Indeed, it is a known fact that $\phi$ induces an isomorphism $J(C) \simeq F(Y)$ between the Jacobian of $C$ and the Fano surface of $Y$ \cite[Theorem~4.8]{Rei72}. 

This shows that the class $w$ corresponds to the class of a rank 2 vector bundle on $C$ of fixed odd degree (we can assume that its degree is 1). Since there is a unique notion of stability on a curve, $\sigma$ must pull-back to slope-stability, and the functor $\phi$ induces an isomorphism between $\sM_\sigma(w)$ and $M_C(2,1)$, the moduli space of slope-stable rank 2 vector bundles of degree 1 on $C$.  Therefore, $M_\sigma(w)$ is an irreducible smooth projective variety of dimension five. As a consequence of \cite[Theorem 4.14(c')]{Rei72}, the intermediate Jacobian $J(Y)$ is isomorphic to the Jacobian of $C$, and under this isomorphism the Abel--Jacobi map coincides with the determinant:
\begin{center}
\begin{tikzcd}[]
M_C(2,1)\dar["\det"] \rar["\sim"] & M_\sigma(w)\dar["\Psi"]\\
\Pic^{1}(C) \rar["\sim"] &J(Y)
\end{tikzcd}
\end{center}
Moreover, fibers of $\det$ (and hence of $\Psi$) are isomorphic to $Y_4$ \cite{new68}.

\subsection{Degree 3}\label{ssec_deg3}

We refer the reader interested in results on smooth cubic threefolds to \cite{Bayer}. Here, the authors show that 
$\sM_\sigma(w)$ is isomorphic to the moduli space $M_G(\kappa)$ of Gieseker-stable sheaves with $\kappa=\left(3,-H,-\frac 12 H^2, -\frac 16 H^3\right)$. The space $M_G(\kappa)$ is smooth and irreducible of dimension 4. $M_G(\kappa)$ maps birationally to the intermediate Jacobian of $Y_3$ and its image is a theta divisor $\Theta$. More precisely, $M_G(\kappa)\to \Theta$ is the blow-up of $\Theta$ in its unique singular point, and the exceptional divisor is the variety $Y_3\subset \sM_\sigma(w)\simeq M_G(\kappa)$ \cite[Theorem 7.1]{Bayer}. 

We also mention the related work \cite{ZYil20}, which compares $\sM_\sigma(w)$ to the Hilbert scheme of skew lines in $Y$.

\begin{remark}[Projectivity of $\sM_\sigma(w)$ for $d\geq 3$]\label{rmk_projectivity345}
We point out that $\sM_\sigma(w)$ is projective for degrees $d\geq 3$. This follows from the explicit descriptions given in these sections: in the case $d=5$, $\sM_\sigma(w)$ is projective since it is a moduli space of quiver representation, and for $d=3,4$ it is projective since it coincides with a space of Gieseker-stable sheaves. 

Alternatively, in all these cases one can apply the criterion \cite[Corollary~1.3]{villalobos2021moishezon}, which together with the Bayer-Macr\'i positivity Lemma \cite[Theorem~4.1]{BM14} implies projectivity of $\sM_\sigma(w)$. Note however that the criterion does not apply in a straightforward way to the cases $d=1,2$, since $\sM_\sigma(w)$ is not normal in those cases. 
\end{remark}


\subsection{Degree 2}
\label{sec_generalitiesY2}
\label{sec_d=2moduli_space}

Recall from Proposition \ref{gieseker} that $\sM_H(w)$ has two irreducible components; one, $\mathcal P$, is a rank 2 projective bundle over $Y$ parametrizing $I_{p|S}$. The complement of $\mathcal P$ is smooth and parametrizes sheaves $\iota_*\sO_S(D)$ associated to roots on hyperplane sections. Points in the intersection are $I_{p|S}$ with $p$ singular in $S$. As illustrated in Section \ref{sec_Fano_threefolds}, $p$ is the singular point of some hyperplane section if and only if $p\in R$, the ramification locus of the double cover $\pi\colon Y\to \pr 3$.
We restate Theorem \ref{thm_Section2} in the case of a general quartic double solid: 

\begin{thm}\label{thm_moduliY2}
Suppose $Y$ is a general quartic double solid. The space $\mathcal{M}_\sigma(w)$ has two irreducible components $\sY$ and $\sC$, both of dimension 3. The component $\mathcal{Y}$ is isomorphic to $Y$, and intersects $\sC$ exactly at the ramification locus $R$. Both components are smooth outside $R$.
\end{thm}

\begin{proof}[Proof of Theorem \ref{thm_moduliY2}]
The statements about the component $\sC$ are already in Theorem \ref{thm_Section2}. In this case, Lemma \ref{lem_numerics} implies that the inclusion $Y \setminus R \hookrightarrow \sM_\sigma(w)$ is an open immersion into a component of $\sM_\sigma(w)$ (since the tangent space at a point $p\in Y\setminus R$ is naturally isomorphic to the tangent space of its image). 

As for the intersection, recall that the components of $\sM_H(w)$ intersect at the points $[I_{p|S}]$, with $p\in S\cap R$ singular in $S$. The wall-crossing morphism $p\colon \sM_H(w)\to \sM_-(w)\simeq \sM_\sigma(w)$ corresponds to the assignment $[I_{p|S}]\to p$, and hence it maps the intersection of the components onto the image of $R$ via $R\subset Y\simeq \mathcal Y$ (see Proposition \ref{prop_GiesekerVSTilt}, Proposition \ref{kuzistilt}). \end{proof}

\subsubsection{Conics on $Y$ and the Abel-Jacobi map}\label{sec_AJY2}


The following construction expresses a relation between $\sC$ and a family of conics on $Y$. Here, $Y$ is any smooth Fano threefold of index 2 and degree 2.


A \textit{conic} on $Y$ is, by definition, a curve $Z\subset Y$ whose image $\pi_{|Z}$ is an isomorphism onto a conic of $\pr 3$. Consider the Hilbert scheme of conics on $Y$, and let $\sC'$ be the component of the Hilbert scheme containing smooth conics (note that since $H$ is not very ample, there may be unusual subschemes of $Y$ with Hilbert polynomial $1+2t$). 

Let $Z\in\sC'$ be a conic;  its image $\pi(Z)$ spans a hyperplane of $\pr 3$ and hence it determines a  hyperplane section $S\subset Y$. Suppose $S$ is smooth. Then, $Z-H_{|S}$ is a root of $S$ since $(Z-H_{|S})\cdot H_{|S}=0$ and $(Z-H_{|S})^2=-2$, so the push forward of the line bundle $\sO_S(Z-H_{|S})$ to $Y$ is in the Gieseker moduli space $\sM_H(w)$, as in Proposition~\ref{gieseker}.
Conversely, given a root $D$ on $S$, the divisor class $D+H_{|S}$ has the Hilbert polynomial of a conic. By Riemann-Roch Theorem and Lemma \ref{lem_KodairaVan}  we see $h^0(\sO_S(D+H_{|S}))=\chi(\sO_S(D+H_{|S}))=2$, and hence $|D+H_{|S}|$ is in a pencil whose general element is a smooth conic. 

Therefore, the assignment $Z\mapsto \iota_*(\sO_S(Z-H_{|S}))$ defines a rational map 
\[ f\colon \sC' \dashrightarrow \sC \]
whose fibers are pencils of rationally equivalent conics 
and which is dominant, since the open locus of $\iota_*\sO_S(D)$ where $D$ is a root of a smooth $S$ is in the image of $f$.
Consider the subscheme of $\sC'$ given by $\sC'_0=\{Z\in \sC' \,|\,Z\subset S\text{ with $S$ smooth}\}$, so that $f$ is defined on $\sC_0'$ 

In the next proposition, we consider the Abel--Jacobi map $\Psi\colon\sC'_0\to J(Y)$ and use diagram \eqref{eq_diagram_AJ_codiff} to study its codifferential (let $W$ be the weighted projective space $\mathbb{P}(1,1,1,1,2)$, so that $N_{Y/W}\simeq \sO_Y(4)$).

\begin{proposition}
\label{prop_AJ_factorsThroughC}
Let $Z$ be a smooth conic in $\sC'_0$. The map $\Psi$ factors through $f\colon\sC'_0\to \sC$ and its codifferential has rank 3 at $Z$.
\end{proposition}

\begin{proof}
The restriction of the Abel--Jacobi map to $\sC'_0$ is constant on pencils of curves, so it factors through $f$ (so in particular $\rk\psi_Z^\vee\leq 3$). Recall that the lower row of \eqref{eq_diagram_AJ_codiff} is part of the long exact sequence associated to 
$$0\longrightarrow N_{Z/Y} \otimes K_Y \longrightarrow N_{Z/W} \otimes K_Y \longrightarrow (N_{Y/W}\otimes K_Y)_{|Z} \longrightarrow 0.$$


The bundle $N_{Z/W}$ can be computed as follows: let $\overline{Z} \subset \pr 3$ be the smooth conic image of $Z$ under $\pi$. The standard exact sequence $ \sO_W(2)\to \sT_W \to \pi^*\sT_{\pr 3} $ yields the diagram 
\begin{equation*}\label{eq_normal_Z_in_E}
\begin{tikzcd}[]
 & \sO_Z(2) \ar[equal]{r} \dar & \sO_Z(2) \dar\\
\sT_Z \rar \ar["\cong"]{d} & \sT_{W|Z} \rar \dar & N_{Z/W} \dar\\
\sT_{\overline{Z}} \rar &  \pi^*\sT_{\pr 3|\overline{Z}}\rar & \pi^*N_{\overline Z/\pr 3}
\end{tikzcd}
\end{equation*}
where all rows and columns are exact. Since $\overline Z$ is a complete intersection we have $N_{\overline Z/\pr 3}\simeq \sO_{\overline Z}(2)\oplus \sO_{\overline{Z}}(1)$, hence from the third column $N_{Z/W} \otimes K_Y \simeq \sO_{Z}^{\oplus 2}\oplus \sO_{Z}(-1)\simeq \sO_{\pr 1}^{\oplus 2}\oplus \sO_{\pr 1}(-2)$.

Moreover, we have $(N_{Y/W}\otimes K_Y)_{|Z}\simeq \sO_{\pr 1}(4)$. The kernel of the connecting homomorphism $\beta_Z\colon H^0((N_{Y/W}\otimes K_Y)_{|Z}) \to H^1(N_{Z/Y} \otimes K_Y)$ has dimension $2-h^0(N_{Z/Y}\otimes K_Y)$, so $\rk \beta_Z \geq 3$.  

Now it suffices to show that the restriction map $$H^0(K_Y(4))\longrightarrow H^0(K_Y(4)_{|Z})$$ is surjective: this implies $\psi_Z^\vee$ has rank at least 3, hence exactly 3.
Since $Z$ lies on a hyperplane section $S$ of $Y$, surjectivity can be checked in two steps: $H^0(K_Y(4))\to H^0(K_Y(4)_{|S})$ is surjective because $H^1(\sO_Y(1))=0$, and the surjectivity of $H^0(K_Y(4)_{|S})\to H^0(K_Y(4)_{|Z})$ follows from the sequence 
$$ \sO_S(2H - Z) \longrightarrow \sO_S(2H) \longrightarrow \sO_Z (2), $$
since $H-Z$ is a root of $S$, and applying Lemma \ref{lem_KodairaVan} to $2H-Z=(H-Z)+H$.
\end{proof}

We summarize the above Proposition \ref{prop_AJ_factorsThroughC} and Remark \ref{rmk_YContractedByAJ} in the following Corollary.

\begin{corollary}\label{cor_C_generically_embedded_by_AJ}
The Abel--Jacobi map $\Psi\colon \sM_\sigma(w) \to J(Y)$ contracts the irreducible component $\mathcal Y$ to a point $y\in J(Y)$ and is a generic embedding on $\sC$.
\end{corollary}


\subsection{Degree 1}\label{sec_d=1}

We reformulate Theorem \ref{introduction_1} in the case of a general double Veronese cone. The arguments are similar to the proof of Theorem \ref{thm_moduliY2}.

\begin{thm}\label{thm_moduliY1}
Suppose $Y$ is a general double Veronese cone. The space $\sM_\sigma(w)$ has two irreducible components $\sY$ and $\sF$, isomorphic respectively to $Y$ and $F(Y)$. The two intersect exactly at the curve $C$, and they are smooth outside of $C$.

The Abel--Jacobi map $\Psi \colon \sM_\sigma(w) \to J(Y)$ contracts $\mathcal Y$ to a singular point $y$ and it is an embedding elsewhere. Moreover, the image $\Psi(\sM_\sigma(w))$ determines $Y$ uniquely.
\end{thm}

\begin{proof}
The statements about the component $\sF\coloneqq \overline{\sM_\sigma(w)\setminus \sY}$ are already in Theorem \ref{thm_Section2}. We can argue as in the proof of Theorem \ref{thm_moduliY2} (with the variation that now $p$ is a singular point of some hyperplane section if and only if $p\in C$) and conclude that $\sY$ and $\sF$ are irreducible components of the desired dimensions, intersecting at $C$ and smooth outside the intersection. 

We are then only left with proving that $\sF$ is isomorphic to $F(Y)$. To this end, we define a morphism $F(Y)\to \sF$ by constructing a family of objects of $\sF$ with base $F(Y)$.  Let $\sH\subset Y\times \mathbb P(H^0(\sO_Y(1)))$ be the universal hyperplane section. Since any line in $Y$ is contained in a unique hyperplane section of $Y$, there is a finite map $F(Y)\to \mathbb P(H^0(\sO_Y(1)))$, so we can define $\sS\subset Y\times F(Y)$ as the pullback of $\sH$. 

Let $\sL\subset\sS\subset Y\times F(Y)$ be the universal line, and consider the sheaf $\sK\coloneqq \sO_\sS(p_Y^\ast H-\sL)$ on $Y\times F(Y)$.
Its fiber is $\sO_S(H-L)$  over a smooth line $L\in F(Y)$. Observe that $H-L$ is a root of $S$, therefore the restriction of $\sK$ to $Y\times (F(Y)\setminus C)$ defines a rational map $F(Y)\dashrightarrow \sM_H(w)$. By the properness of $F(Y)$ and $\sM_H(w)$, the above map extends to a morphism $F(Y)\to \sM_H(w)$, which is an isomorphism of $F(Y)\setminus C$ onto the smooth locus of the irreducible component of $\sM_H(w)$ parametrizing roots on hyperplane sections. Let $\sK'$ be the pull-back of the universal family to $Y\times F(Y)$ which is isomorphic to $\sK$ on $Y\times (F(Y)\setminus C)$. The fiber of $\sK'$ at a singular line in $C\subset F(Y)$ is the flat limit of a family of push forwards of roots, and therefore is of the form $I_{p|S}$ for some $p\in S$ and some hyperplane section $S$. 

Then, the object $\sK''\coloneqq \mathbb R_{\sO_Y(-1)}^{F(Y)}(\sK)\in \!^\perp\pair{\sO_Y(-1)}_{F(Y)}$ (notation as in Section~\ref{sec_basechange_mutations}) is a family of objects of $\sF$ parametrized by $F(Y)$, and it induces the desired morphism $\mu\colon F(Y)\to \sF$.

Moreover, $\sK''$ is isomorphic to $\sK$ off the curve $C$, and its fiber over a singular line $[L]\in C$ fits in a triangle
\[ \mathcal \sK''_{[L]} \to I_{p|S} \xrightarrow{\mathrm{coev}} \sO_Y(-1)[1]\oplus \sO_Y(-1)[2] \]
whence $\sK_{[L]}\simeq E_p$ (as in the proof of Proposition~\ref{prop_GiesekerVSTilt}). Therefore, $\sK''$ is a universal family for $\sF$ supported on $F(Y)$. Then, the morphism $\mu$ is an isomorphism.

Lastly, the component $\sY$ is contracted by $\Psi$ as in Remark \ref{rmk_YContractedByAJ}, while the Abel--Jacobi mapping is an embedding on smooth lines by \cite[\S 10]{Tih82}, and $Y$ is determined uniquely via $\Psi(\sM_\sigma(w))=\Psi(F(Y))$ by \cite[Thm.~12]{Tih82}.
\end{proof}

\section{A categorical Torelli theorem for quartic double solids}
\label{sec_Torelli}

 As an application of the previous construction, this section is dedicated to the proof of Theorem \ref{introduction_3}, we state it here for the reader's convenience. In this section, $d=2$.
 
\begin{thm}\label{thm_torelliY2}
Let $Y$ and $Y'$ be two general quartic double solids such that there exists an equivalence $u \colon \Ku(Y') \xrightarrow{\sim} \Ku(Y)$. Then $Y'\simeq Y$.
 \end{thm}
 
This theorem improves a result of Bernardara and Tabuada \cite[ Corollary~3.1 (iii)]{BT16} who show that the same holds under the additional assumption that $u$ is of Fourier--Mukai type. The authors use noncommutative motives to construct isogenies between intermediate jacobians. In \cite{perry2020integral}, the author obtains similar results on intermediate jacobians, using only cohomological methods. Our technique is inspired by that of \cite{BMMS12}; we construct an isomorphism between moduli spaces and use the Abel--Jacobi map to argue that this is sufficient to conclude.

\subsection{Images of stable objects} First, we prove that given an equivalence as in Theorem \ref{introduction_3}, we can produce another equivalence with nice properties on the objects parametrized by $\mathcal{M}'$. 

We start with a series of Lemmas, and denote by $\mathcal{A}= \Ku(Y) \cap \Coh^0_{\alpha,-\frac{1}{2}}(Y)$ the heart of $\Ku(Y)$ associated to the stability condition $\sigma$ (Proposition~\ref{thm_stab_con_on_Ku}). We say that a heart of a bounded t-structure $\sB\subset \Ku(Y)$ is \textit{$\tau$-invariant} if $\tau(\sB)=\sB$. Then we have:

\begin{lemma}
The heart $\sA$ is $\tau$-invariant.
\end{lemma}

\begin{proof}
By definition, $\mathcal{A}= \Ku(Y) \cap \Coh^0_{\alpha,-\frac{1}{2}}(Y)$. Clearly $\tau(\Ku(Y))\subseteq \Ku(Y)$ since up to shift $\tau$ is the Serre functor of $\Ku(Y)$. Then, it suffices to show that $\tau$ preserves $\Coh^0_{\alpha,-\frac{1}{2}}(Y)$ to conclude.

Observe that $\tau$ acts trivially on numerical classes of objects in $D(Y)$, since it preserves the class $H$. Therefore, $\tau$ preserves slopes. Moreover, it preserves exact sequences of sheaves and hence slope stability. This implies that $\tau$ preserves $\Coh^\beta(Y)$ for all $\beta\in \R$. Similarly one sees that $\tau$ preserves $\sigma_{\alpha,\beta}$-semistable objects for all $(\alpha,\beta)\in \R_{>0}\times \R$, and hence it preserves the double tilt. 
\end{proof}

\begin{lemma}\label{extdimension}
The homological dimension of a $\tau$-invariant heart $\mathcal{B}\subset \Ku(Y)$ is equal to 2, \textit{i.e.} $\Ext^i(E,F)=0$ for $i\neq 0,1,2$ for $E, F\in \mathcal{B}$.
\end{lemma}
\begin{proof}
Since $\mathcal{B}$ is the heart of a bounded t-structure, we have by definition that $\Ext^i(E,F)=0$ for any $E,F \in \mathcal B$, $i<0$. Now, using the above Serre functor we get:
\[
\Ext^i(E,F)=\Hom(E,F[i])= \Hom(F[i],S_{\Ku(Y)}(E))= \Ext^{2-i}(F ,\tau E),
\]
which concludes the proof, since $\sB$ is $\tau$-invariant and hence $\tau E \in \mathcal{B}$.
\end{proof}

 \begin{lemma}
 Let $u \colon \Ku(Y') \xrightarrow{\sim} \Ku(Y)$ be an equivalence of categories, and let $E\in \Ku(Y')$ be an object of class $w$; then up to a sign, we have $[u(E)]=w$ or $[u(E)]=2v-w$.
 \end{lemma}
 \begin{proof}
Since $u$ is an equivalence of categories, it must preserve dimensions of $\Hom$ spaces, so in particular $\chi(u(E),u(E))=\chi(E,E)=-2$. From the intersection matrix of the Euler form given in section 2, if $[u(E)]=av+bw$, then $\chi(u(E),u(E))=-(a+b)^2-b^2$. The only possible pairs $(a,b)$ are $(0,\pm 1)$ and $(\mp 2,\pm 1)$. Finally, since an equivalence of categories induces a homomorphism on the Grothendieck groups, we can choose the sign uniformly for all objects.
\end{proof}
\begin{lemma}\label{dimext1} 
Let $0\neq C\in \mathcal{A}\subset \Ku(Y)$; then $\dim \Ext^1(C,C)\geq 2$.
\end{lemma}
\begin{proof}
Since $\sA$ is the heart of a stability condition, $C\neq 0$ implies that $[C]\neq 0$. Then, we have $\chi(C,C) \leq -1$, $\dim \Hom(C,C)\geq 1$, and $\Ext^3(C,C)=0$ from Lemma \ref{extdimension}. Therefore,  $\dim \Ext^1(C,C)\geq 2$.
\end{proof}
\begin{lemma}\label{ext1=3}
Let $C \in \Ku(Y)$ with $\dim\Ext^1(C,C)=3$; then up to a shift $C \in \mathcal{A}$, and if additionally $[C]=  w$ then $C$ is stable.
\end{lemma}
\begin{proof}
Consider the spectral sequence for objects in $\Ku(Y)$ whose second page is given by
\begin{equation}\label{spectral}
E^{p,q}_2 = \bigoplus_i \Hom^p(\mathcal{H}^i(C), \mathcal{H}^{i+q}(C)) \Longrightarrow \Hom^{p+q}(C,C)
\end{equation}
(see \cite[Lemma 4.5]{BMMS12}), where the cohomology is taken with respect to the heart $\mathcal{A}$. Since by Lemma~\ref{extdimension} the Ext-dimension of $\mathcal{A}$ is 2, it follows that $E^{1,q}_2=E^{1,q}_\infty$, so that if we take $q=0$, by Lemma~\ref{dimext1} we get
\[
3=\hom^1(C,C) \geq \bigoplus_i \hom^1(\mathcal{H}^i(C), \mathcal{H}^{i}(C)) \geq 2r,
\]
where $r>0$ is the number of non-zero cohomology objects of $C$. Then $r=1$ and $C\in \mathcal{A}$ up to a shift; if we also assume $[C]=w$, then $C$ must be stable since $w$ has maximal slope and is primitive.
\end{proof}
\begin{lemma} \label{keylemma}
Let $C \in \Ku(Y)$ of class $[C]=w$ satisfying:
\[
\hom^i(C,C)=\begin{cases}1 \qquad \text{if }\ i=0 \\ 4 \qquad \text{if }\ i=1 \\ 1 \qquad \text{if }\ i=2 \\ 0 \qquad \text{if }\ i=3\end{cases}
\]
then, up to a shift, $C \in \mathcal{A}$, and in particular $C$ is stable.
\end{lemma}
\begin{proof}
As in the proof of Lemma \ref{ext1=3}, we get that
\[ 4=\hom^1(C,C) \geq \bigoplus_i \hom^1(\mathcal{H}^i(C), \mathcal{H}^{i}(C)) \geq 2r,\]
which in this case yields $r=1$ or $r=2$. If $r=1$, we're done. Assume for the sake of contradiction that $r=2$, and call $M$ and $N$ the two non-zero cohomology objects of $C$. The objects $M$ and $N$ must be adjacent cohomologies: otherwise $C$ is a sum of their shifts because $\Ext^{i}(M,N)=0$ for $i>2$ (by Lemma \ref{extdimension}), but this contradicts $\hom^0(C,C)=1$. We let $M= \mathcal{H}^{j}(C)$ and $N= \mathcal{H}^{j\pm 1}(C)$.

Since $\chi(F,F)\leq -1$ for all non-zero objects of $\sA$, one necessarily has $\hom^1(M,M)=\hom^1(N,N)=2$, and therefore $\hom^0(M,M)=\hom^0(N,N)=1$ and $\hom^2(M,M)=\hom^2(N,N)=0$. In particular, $\chi(M,M)=\chi(N,N)=-1$. This means that $[M]$ and $[N]$ can only be equal to $-v$ or $w-v$ (up to a sign depending on the parity of $j$). Since $[M]+[N]= w$, the only possibilities are $[M]=-v$, $[N]=w-v$, or the converse.  
In any case, since the Euler form is symmetric and $\chi(- v,w-v)=0$, we have $\chi(M,N)=0$. 

Now we gather all this information together in the second page of the spectral sequence (\ref{spectral}). It has zeros everywhere except for the spaces:
\begin{center}
$E_2^{p,q}$\quad = \quad\begin{tabular}{|ccc}
0 & 0 & 0 \\
 $\Hom(N,M)$ & $\Hom^1(N,M)$ & $\Hom^2(N,M)$\\
$\Hom(M,M)\oplus \Hom(N,N)$ & $\Hom^1(M,M)\oplus \Hom^1(N,N)$ &0\\
\hline
 $\Hom(M,N)$ & $\Hom^1(M,N)$ & $\Hom^2(M,N)$\\
\end{tabular} 
\end{center}
Therefore, the whole sequence degenerates at page 3, and the central column survives in the limit. We denote with $a,b,c,d,e,f$ the dimensions of the spaces in the first and third row, the table of dimensions for page 2 is:
\begin{center}
$\dim E_2^{p,q}$\quad = \quad \begin{tabular}{|ccc}
0 & 0 & 0 \\
 $a$ & $b$ & $c$\\
 2 & 4 &0\\
\hline
 $d$ & $e$ & $f$\\
\end{tabular}
\end{center}
Our assumptions on $C$ constrain the dimensions on the diagonals of page 2: one sees right away that we must have $a=0$, $b=1$, $c=0$ and $d=0$, and a surjection $\Hom(M,M)\oplus \Hom(N,N) \twoheadrightarrow \Hom^2(M,N)$ with kernel of dimension $g \coloneqq 2-f$. On the other hand, we also have $e=f$ since the alternate sum of dimension on the row $q=-1$ is $\chi(M,N)=0$. This implies that $1=\hom(C,C)=g+e=g+f=2$, which is a contradiction. Then, the only possibility is that $r=1$, and $C$ only has one cohomolgy object and therefore belongs to (a shift of) $\sA$.
\end{proof}

The category $\Ku(Y)$ admits an autoequivalence called the \textit{rotation functor} $$\mathsf R(-)\coloneqq \mathbb L_{\sO_Y} ( - \otimes \sO_Y(1))$$
(see \cite[Section~3.3]{Kuz15_fractional}, in particular \cite[Corollary~3.18]{Kuz15_fractional} for the proof of the fact that $\mathsf R$ is an autoequivalence). 


\begin{lemma}\label{numfix}
The map induced on $K_{num}(\Ku(Y))$ by the rotation functor maps $w$ to $w-2v$ and $v$ to $w-v$.
\end{lemma}
\begin{proof}
To prove the first statement, we compute $\mathsf R(E_p)$ since $[E_p]=w$. Twisting the sequence \eqref{eq_def_of_Ep} and mutating across $\sO_Y$, we see that $\mathsf R(E_p)=\mathbb L_{\sO_Y} (I_p(1))$. The latter is computed by the triangle
\begin{equation}
    \label{eq_def_of_Mp} \sO_Y^{3} \to I_p(1)\to \mathbb L_{\sO_Y} (I_p(1)),
\end{equation}
and it has class $[I_p(1)]-3[\sO_Y]=2v-w$ in $K_{num}(\Ku(Y))$. Arguing similarly, one considers the ideal sheaf of a line $l$, which has class $v$, and checks that $[L_{\sO_Y} (I_p(1))]=[I_l(1)]-2[\sO_Y]=w-v$.
\end{proof}

Now if we take an object $E \in \mathcal{M}'$, then $u(E)$ has class $\pm w$ or $\pm(2v-w)$. In the latter case we can replace $u$ with $u\circ \mathsf R$ so that $u(E)$ has class $\pm w$.  Lemmas \ref{ext1=3} and \ref{keylemma} show that there exists an integer $n$ such that $u(E)[n] \in \mathcal{A}$ and is stable of class $w$ (in particular, it has phase 1). 

We want to prove that the shift can be taken uniformly, \textit{i.e.} that, for the same $n$ and any other object $F$ in $\mathcal{M}'$, we have $u(F)[n]\in\sA$. For all such $F$, we must have $\Hom(E,F[1])=\Ext^1(E,F)\neq 0$, since $\Hom(E,F)=0$ by stability, $\Ext^3(E,F)=\Ext^{-1}(F,E)=0$ by use of the Serre functor on $\Ku(Y')$, and $\chi(E,F)=-2$ because of their numerical class. Likewise, $\Hom(F,E[1])\neq 0$. Applying $u(-)[n]$ we get $\Hom(u(E)[n],u(F)[n+1])\neq 0$ and $\Hom(u(F)[n],u(E)[n+1])\neq 0$. Now, $u(E)[n]$ has phase 1, and let $\phi \in \Z$ be the phase of $u(F)[n]$: from the above we get $1< \phi +1$ and $\phi<2$, hence $\phi=1$ and $u(F)[n]\in \sA$.

We can thus summarize the results of this section with the following roposition:

\begin{proposition}\label{map_given_by_equiv}
The equivalence $u \colon \Ku(Y') \xrightarrow{\sim} \Ku(Y)$ induces a bijection on closed points $u\colon \mathcal{M}'(k) \to \mathcal{M}(k)$.
\end{proposition}
\begin{proof} We have shown that the assignment $E\mapsto u(E)$ sends points of $\sM'$ to points of $\sM$ (and similarly, $u^{-1}$ takes $\sM$ to $\sM'$). Since $u$ is an equivalence, the induced map is a bijection on closed points. 
\end{proof}

\subsection{Universal families and convolutions}\label{sec_univ_fam_and_convolutions}
Let $\sM$ and $\sM'$ denote the moduli spaces of objects of class $w$ in $\Ku(Y)$, resp. $\Ku(Y')$. As usual, we will denote by $\mathcal{Y}$ the component isomorphic to $Y$ in $\sM$ (we use the same convention in the case of $Y'$). 
The component $\mathcal Y$ admits a universal family, whose construction we outline here (see also the proof of Proposition~\ref{prop_GiesekerVSTilt}). Let $\sI \in D^b(Y \times \mathcal Y)$ be the pull back of the ideal sheaf of the diagonal via the isomorphism $Y\times \mathcal  Y \xrightarrow{\sim} Y\times Y$. Then, define $\sE\coloneqq \mathbb R_{\sO_Y(-1)}^{\mathcal Y}(\sI)[-1]$ to be the projection to $\!^\perp\pair{\sO_Y(-1)}_\sY$ (notation as in Section~\ref{sec_basechange_mutations}). It is the desired universal family: for example the restriction of $\sE$ to a fiber $Y\simeq Y\times \set{s}$ above a closed point $s\in S$ fits in a triangle
\[ I_s \to \sO_Y(-1)[2] \to \sE_s[1], \]
whence $\sE_s=E_s$ (this triangle is \eqref{eq_def_of_Ep}). Likewise, define $\sE'$ to be the universal family above $\sY'$. 


Define the composite functor 
\[
 F\colon D(Y')\xrightarrow{\rho'} \Ku(Y') \xrightarrow{u}\Ku(Y) 
 \xrightarrow{\epsilon} D(Y)
 \]
where $\rho'$ is the natural projection, $\epsilon$ the full embedding, and $u$ an exact equivalence. 

If $F$ is a Fourier--Mukai functor, \textit{i.e.}, $F\simeq \Phi_{\mathcal G}$ for some $\mathcal G\in D^b(Y\times Y')$, then one defines 
$$ \Phi_{\mathcal G}\times \id_{\sY'} \coloneqq \Phi_{\mathcal G \boxtimes \sO_{\Delta_{\sY'}}} \colon D^b(Y'\times \sY') \to D^b(Y\times \sY') $$
and the object 
$$\tilde\sE\coloneqq \Phi_{\mathcal G}\times \id_{\sY'} (\sE') \in D^b(Y\times \sY')$$
is a family of objects of $D(Y)$ parametrized by $\sY'$, which defines a morphism $Y'\simeq \sY'\to \sM$. 

If $F$ is not a Fourier--Mukai functor, then one can use \emph{convolutions} as in \cite[Section~5.2]{BMMS12}, whose method applies without changes to our case, and produces a family $\tilde \sE \in D^b(Y\times \mathcal Y')$. 

In either case, $\tilde \sE$ can be used to show that the map defined in Proposition~\ref{map_given_by_equiv} is a morphism which restricts to an isomorphism $Y'=\sY'\simeq \sY=Y$. For a closed point $s\in\sY'$, denote by $i_s$ (resp. $i'_s$) the inclusion $Y\times \set{s} \to Y\times \sY'$ (resp. $Y'\times \set{s} \to Y'\times \sY'$). Then we have:

\begin{lemma}
For any closed point $s\in \sY'$ we have $i_s^*(\tilde\sE)\simeq F((i'_s)^*\sE')$.
\end{lemma}

\begin{proof}
This is \cite[Lemma 5.3]{BMMS12}.
\end{proof}

\begin{proof}[{Proof of Theorem \ref{thm_torelliY2}}]
Whether $\tilde\sE$ is constructed with a Fourier--Mukai functor or by means of convolutions, one sees that $i_s^*(\tilde\sE)\simeq F((i'_s)^*\sE')=u(E'_s)$ for all $s\in\sY'$. In other words, $\tilde\sE$ is a family of objects of $\sM$ parametrized by $\sY'$, and it yields a proper morphism $\alpha\colon \sY'\to \sM$ with the property that $\alpha(s)\in \sM$ corresponds to the object $u(E'_s)$ for all $s\in \sY'$. Since $\sY'$ is irreducible, $\alpha$ factors through one of the components of $\sM$. We claim that $\alpha$ must factor through $\sY$. Granting the claim for a moment, smoothness of $\sY$ implies that $\alpha$ is the desired isomorphism $\sY'\xrightarrow{\sim}\sY$.

To establish the claim, observe that $\alpha$ is, in particular, a morphism dominating one of the components of $\sM$ and $\alpha$ is birational onto its image. Since $\sY'$ is rationally connected and $\alpha$ preserves this property (see for example \cite{Kol96}), its image cannot lie in $\sC$, which is not rationally connected as a consequence of Corollary \ref{cor_C_generically_embedded_by_AJ}.
\end{proof} 

\begin{remark}[Generality assumption]
The theorem should still hold assuming that $Y$ and $Y'$ are smooth, but possibly not general. In our proof, we use Corollary \ref{cor_C_generically_embedded_by_AJ}, which holds for any smooth $Y$, to distinguish the two components $\sY'$ and $\sC'$. However, if $Y$ is not general we cannot rule out that the wall-crossing of Proposition~\ref{prop_GiesekerVSTilt} introduces  additional components in $\sM_\sigma(w)$.
\end{remark}


\newcommand{\etalchar}[1]{$^{#1}$}

\end{document}